\documentclass{amsart}
\usepackage{amsmath}
\usepackage{amssymb}
\usepackage{amsthm}
\usepackage{tikz-cd}
\RequirePackage[l2tabu, orthodox]{nag}
\usepackage{url}
\usepackage{hyperref}
\usepackage{color}
\hypersetup{nesting=true,debug=true,naturalnames=true}
\usepackage{mathtools}
\DeclareMathOperator{\id}{id}
\DeclareMathOperator*{\colim}{colim}

\newtheorem{theorem}[subsection]{Theorem}
\newtheorem*{maintheorem}{Main Theorem}
\newtheorem{lemma}[subsection]{Lemma}
\newtheorem{prop}[subsection]{Proposition}
\newtheorem{cor}[subsection]{Corollary}
\newtheorem{conj}[subsection]{Conjecture}
\theoremstyle{definition}
\newtheorem{defi}[subsection]{Definition}
\newtheorem{terminology}[subsection]{Terminology}
\newtheorem{notation}[subsection]{Notation}
\newtheorem{assumption}[subsection]{Assumption}
\theoremstyle{remark}
\newtheorem{remark}[subsection]{Remark}
\newtheorem*{example}{Example}

\newcommand{\groundring}{R}

\newcommand{\compositionproduct}{\mathbin{\mkern-1mu\triangleleft}}
\newcommand{\FF}{\mathbb{F}}
\newcommand{\QQ}{\mathbb{Q}}
\newcommand{\ZZ}{\mathbb{Z}}

\urldef\geoffroyhomepage\url{geoffroy.horel.org}
\urldef\gabrielhomepage\url{drummondcole.com/gabriel/academic/}
\newcommand{\mailurl}[1]{\email{\href{mailto:#1}{#1}}}

\title{Homotopy transfer and formality}
\author{Gabriel C. Drummond-Cole}
\address{Center for Geometry and Physics\\ Institute for Basic Science (IBS)\\\newline
77 Cheongam-ro, Nam-gu, Pohang-si, Gyeongsangbuk-do 37673\\
Republic of Korea}
\mailurl{gabriel.c.drummond.cole@gmail.com}
\urladdr{\gabrielhomepage} 
\thanks{Drummond-Cole was supported by IBS-R003-D1. Horel was supported by ANR-18-CE40-0017 PerGAMo, funded by Agence Nationale  pour  la  Recherche}

\author{Geoffroy Horel}
\address{Institut Galil\'ee\\
Universit\'e Sorbonne Paris Nord\\
\newline 
99 avenue Jean-Baptiste Cl\'ement, 93430 Villetaneuse, France}
\address{D\'epartement de math\'ematiques et applications, \'Ecole normale sup\'erieure,
\newline 
45 rue d'Ulm, 75230 Paris Cedex 05, France}
\mailurl{horel@math.univ-paris13.fr}
\urladdr{\geoffroyhomepage}

\begin{document}
\begin{abstract}
In~\cite{CiriciHorel:MHSFSMF,CiriciHorel:ECPFTC}, the second author and Joana Cirici proved a theorem that says that given appropriate hypotheses, $n$-formality of a differential graded algebraic structure is equivalent to the existence of a chain-level lift of a homology-level degree twisting automorphism using a unit of multiplicative order at least $n$. 

Here we give another proof of this result of independent interest and under slightly different hypotheses. We use the homotopy transfer theorem and an explicit inductive procedure in order to kill the higher operations. As an application of our result, we prove formality with coefficients in the $p$-adic integers of certain dg-algebras coming from hyperplane and toric arrangements and configuration spaces.
\end{abstract}

\maketitle

\section*{Introduction}
An algebraic structure $A$ (e.g. an associative algebra, a commutative algebra, an operad, etc.) in the category of chain complexes is said to be formal if it is connected to its homology $H(A)$ by a zig-zag of quasi-isomorphisms that preserve the algebraic structure. 
There are many interesting examples from a variety of arenas, including rational homotopy theory~\cite{Sullivan:ICT,FelixHalperinThomas:RHT}, Kontsevich formality~\cite{Kontsevich:DQPM,Tamarkin:FCOLD}, K\"ahler manifolds~\cite{DeligneGriffithsMorganSullivan:RHTKM}, string topology of complex projective spaces with integral coefficients~\cite[Theorem 4.3]{berglundborjeson:KAAFLH}, etc.

Let us generically use the term algebra to refer to any type of algebraic structure such as those mentioned above; we assume that the structure operations are all degree zero. For any algebra $A$, and any unit $\alpha$ of the base ring, one can construct an automorphism $\sigma_{\alpha}$ of $H(A)$ that is given in homological degree $n$ by multiplication by $\alpha^n$. By a standard homotopical algebra reasoning, if $A$ happens to be formal, this automorphism can be lifted to an endomorphism of $A$ (or at least an endomorphism of a cofibrant replacement of $A$). 

It was observed in the introduction of~\cite{DeligneGriffithsMorganSullivan:RHTKM} (see also the last section of \cite{Petersen:GTFLD}) that the converse should be true if $\alpha$ is of infinite order. The intuition is the following : if such an endomorphism exists at the level of chains, then any higher Massey product has to be compatible with this action but then we see that they have to be zero because they intertwine multiplication by $\alpha^n$ with multiplication by $\alpha^m$ with $n\neq m$. As stated, this is not rigorous : for example, without some coherence assumption, the vanishing of the classical Massey products of a differential graded algebra is not a sufficient condition to ensure formality. 
In fact Deligne--Griffiths--Morgan--Sullivan proved their result using a different method. 
Sullivan proved a statement of this kind for differential graded algebras in characteristic zero~\cite[Theorem 12.7]{Sullivan:ICT} which was generalized to other algebraic structures by Guill\'en Santos--Navarro--Pascual--Roig~\cite[Theorem 5.2.4]{GuillenSantosNavarroPascualRoig:MSFO}. 
The method in both cases uses a chain level filtration and does not explicitly pass through Massey products. 

One of our goals in this paper is to give a method for proving formality by making more precise the intuition explained in the previous paragraph. We use the fact that the Massey products are the shadow of the $P_\infty$-structure on the homology of the algebra. 
Unlike the Massey products, the $P_\infty$-structure is not uniquely defined but it is unique in a suitable homotopical manner. 
Moreover, it contains all the homotopical information of the algebra. Our main result is the following :
\begin{maintheorem}
\label{theorem: main}
Let $R$ be a commutative ring. Let $P$ be an operad in the category of $R$-modules. Let $(A,d)$ be a $P$-algebra in the category of differential graded $\groundring$-modules such that the chain complex $(H(A,d),0)$ can be written as a homotopy retract of $(A,d)$. 
Let $\alpha$ be a unit in $\groundring$ and let $\hat{\sigma}$ be an endomorphism of the $P$-algebra $(A,d)$ such that the induced map on $H(A,d)$ is the degree twisting by $\alpha$ (see Definition~\ref{defi: degree twisting}). 
\begin{itemize}
\item If $\alpha^k-1$ is a unit of $\groundring$ for $k\leq n$, then $(A,d)$ is $n$-formal as a $P$-algebra.
\item If $\alpha^k-1$ is a unit of $\groundring$ for all $k$, then $(A,d)$ is formal as a $P$-algebra.
\end{itemize}
\end{maintheorem}

The assumption that $(H(A,d),0)$ can be written as a homotopy retract of $(A,d)$ allows us to apply the homotopy transfer theorem (Theorem \ref{theorem:HPL}). It is automatic if $\groundring$ is a field or more generally if $A$ and $H(A)$ are degreewise projective and $\groundring$ is a hereditary ring (a ring is hereditary if a submodule of a projective module is projective). Let us mention that Dedekind rings (in particular principal ideal domains) are hereditary.
These conditions can be weakened a little further. See Remark~\ref{remark: generalization of homotopy retract}.

The main improvement over the kind of classical results of~\cite{Sullivan:ICT} and the improvement of~\cite{GuillenSantosNavarroPascualRoig:MSFO} is the extension to give $n$-formality results outside characteristic zero fields. 
The method of proof in both references relies on a filtration that fails as soon as $\alpha$ does not have infinite order.

One reason this result is interesting is because the theory of \'etale cohomology gives algebras with automorphisms of this type. Let us explain this with a simple example. We take $A$ to be the algebra $C^*(\mathbf{P}^n_{\mathbb{C}},\mathbb{Q}_\ell)$ of singular cochains on the complex projective space with $\mathbb{Q}_\ell$-coefficients. Standard results of \'etale cohomology imply that this algebra is quasi-isomorphic to the algebra $B=C^*_{et}(\mathbf{P}^n_{\overline{\mathbb{Q}}},\mathbb{Q}_\ell)$ of \'etale cochains of the projective space over the algebraic closure of $\mathbb{Q}$. Since the projective space is actually defined over $\mathbb{Q}$, the algebra $B$ has an action of the absolute Galois group of $\mathbb{Q}$. If we fix a prime $p$ different from $\ell$, we can pick a lift $\sigma$ of the Frobenius of $\overline{\mathbb{F}}_p$ in the absolute Galois group of $\mathbb{Q}$ and it can be shown that $\sigma$ acts on $H^{2k}(\mathbf{P}^n_{\mathbb{C}},\mathbb{Q}_\ell)$ by multiplication by $p^k$. We are thus exactly in the situation of the theorem above and we deduce
that $B$ is formal, and thus that $A$ is as well.

We could apply the same strategy with the algebra of cochains of complex projective space with $\mathbb{Z}_\ell$-coefficients. In that case, the Frobenius lift still acts in degree $2k$ by multiplication by $p^{k}$. However, contrary to what happens in characteristic zero, $p^{\ell-1}-1$ is not a unit in $\mathbb{Z}_\ell$. The above theorem lets us conclude that $C^*_{et}(\mathbf{P}^n_{\overline{\mathbb{Q}}},\mathbb{Z}_\ell)$ is $(\ell-2)$-formal (it is in fact $(2\ell-4)$-formal by the variant \ref{theo: variant}).

A similar suite of results was proven in the papers~\cite{CiriciHorel:MHSFSMF} and~\cite{CiriciHorel:ECPFTC} by Cirici and the second author. There the method used was different and relied on deep results in abstract homotopy theory, most notably, Hinich's recent result comparing Lurie's $\infty$-categorical approach to algebras over an operad with the model categorical approach (see~\cite{Hinich:RAM}). 
Here the methods used are comparatively easier and more explicit. We can in fact write an inductive formula for a formality quasi-isomorphism. 
Moreover, we are able to improve one of the results of \cite{CiriciHorel:ECPFTC} by removing a simple connectivity hypothesis and allowing the coefficient ring to be more general than a field.
From this we obtain a result of partial formality for complements of hyperplane arrangements and toric arrangements with coefficients in the $p$-adic integers that we believe is new. Let us mention however, that not all of the results of~\cite{CiriciHorel:ECPFTC} can be recovered from the methods of our paper. Most notably, the result of $(p-2)$-formality of the little disks operad with coefficients in $\mathbb{F}_p$ proved in~\cite[Theorem 6.7]{CiriciHorel:ECPFTC} is not a consequence of our main theorem.

\subsection*{Structure of the paper}
In section~\ref{sec: reminder} we review mostly standard conventions, definitions, and facts about operadic homotopy algebra. We briefly review operads and cooperads, algebras and coalgebras, coderivations, homotopy algebras, and homological perturbation.
The only things that are non-standard are the following:
\begin{enumerate}
	\item we use the symbol $\compositionproduct$ instead of $\circ$ for the composition product of $\mathbb{S}$-modules and $\mathbb{N}$-modules, and
	\item we use the terminology \emph{component} of a homotopy algebra in a non-standard way---see Terminology~\ref{terminology: components} and Remark~\ref{remark: warning}.
	\item we define \emph{$n$-formality} in Definition~\ref{defi: n-formality} and connect it to formality in Propositions~\ref{prop: chains n-formality implies formality} and~\ref{prop: cochains n-formality implies formality}.
\end{enumerate}

Then Section~\ref{section: induction} constitutes the proof of the main theorem. 
The method is to construct a sequence of isomorphisms of homotopy algebras which witness the coherent vanishing of successively more and more of the higher structure operations.

Section~\ref{section: variant theorem} presents a simple algebraic variant of the main theorem. Both this variant and the main theorem are applied in 
Section~\ref{sec: applications} to yield the following examples.
\begin{enumerate}
	\item \label{item: complex algebraic varieties} Formality of complex algebraic varieties whose mixed Hodge structure is pure of some weight,
	\item \label{item : little disks} Formality of the little disks operad,
	\item \label{item: hyperplane arrangements} Formality of complements of hyperplane arrangements,
	\item \label{item: toric arrangements} Formality of complements of toric arrangements and
	 
	\item \label{item: configurations} Coformality of the space of configurations of points in Euclidean space.
\end{enumerate}
To the best of our knowledge, the results obtained with integral coefficients in examples~\eqref{item: hyperplane arrangements},~\eqref{item: toric arrangements} and~\eqref{item: configurations} are new.
\subsection*{Conventions}
Fix a commutative ground ring $\groundring$. 
All tensor products are taken over $\groundring$ unless otherwise specified.
When working with symmetric operads we insist that all prime numbers are invertible in $\groundring$ (i.e. $\groundring$ is a $\mathbb{Q}$-algebra). The symmetric group on the set $\{1,\ldots, n\}$ is denoted $\mathbb{S}_n$.

\begin{defi}
\label{defi: degree twisting}
Let $V$ be a graded $\groundring$-module.
Let $\alpha$ be a unit in $\groundring$.
The \emph{degree twisting} by $\alpha$, denoted $\sigma_\alpha$, is the linear automorphism of $V$ which acts on the degree $n$ homogeneous component of $V$ via multiplication by $\alpha^n$.
\end{defi}

\section{Reminder and conventions on operadic homotopy algebra}
\label{sec: reminder}
\subsection{Operads}
A \emph{$\mathbb{N}$-module} is a collection $\{P(n)\}$ indexed by $n\ge 0$ of chain complexes over $\groundring$. 
A \emph{$\mathbb{S}$-module} $P$ is a collection $\{P(n)\}$ indexed by $n\ge 0$ of right ($\groundring$-linear, differential graded) $\mathbb{S}_n$-representations.
The index $n$ is called the \emph{arity}.
When we are working with $\mathbb{S}$-modules, we insist that our ground ring is a field of characteristic zero (so that we can, e.g., identify invariants and coinvariants of symmetric group actions).

Maps of $\mathbb{N}$-modules (respectively $\mathbb{S}$-modules) are collections of (equivariant) chain maps.
There are monoidal products on $\mathbb{N}$-modules and $\mathbb{S}$-modules defined as follows (these products are often denoted $\circ$ in the literature; we avoid this because of the potential for confusion):
\begin{gather*}
(P\compositionproduct Q)(n) 
=
\bigoplus_{k=0}^\infty 
\bigoplus_{n_1+\cdots+n_k=n}P(k)\otimes Q(n_1)\otimes\cdots\otimes Q(n_k).
\\
\end{gather*}
\begin{multline*}
(P\compositionproduct Q)(n)
=\\
\bigoplus_{k=0}^\infty 
P(k)\otimes_{\groundring[\mathbb{S}_k]}\left(\bigoplus_{n_1+\cdots+n_k=n}Q(n_1)\otimes\cdots\otimes Q(n_k)\otimes_{\groundring[\mathbb{S}_{n_1}\times \cdots\times \mathbb{S}_{n_k}]}\groundring[\mathbb{S}_n]\right).
\end{multline*}
The unit $I$ has a rank one free $\groundring$-module in $I(1)$ and the zero representation elsewhere.
We suppress the associator isomorphisms throughout.

An \emph{operad} is a monoid in the monoidal category of $\mathbb{S}$-modules. 
A \emph{cooperad} is a comonoid in this monoidal category. 
A \emph{non-symmetric operad} (respectively \emph{cooperad}) is a monoid (\emph{comonoid}) in the monoidal category of $\mathbb{N}$-modules. The unit $I$ with respect to $\compositionproduct$ is thus both a cooperad and an operad. We refer to it as the trivial cooperad and the trivial operad.
A \emph{coaugmentation} of a (possibly non-symmetric) cooperad $C$ is a cooperad map from the trivial cooperad to $C$; an augmentation of a (possibly non-symmetric) operad $P$ is an operad map from $P$ to the trivial operad.
A \emph{weight-grading} on a cooperad or operad is an extra $\mathbb{N}$-grading on the underlying $\mathbb{N}$-module or $\mathbb{S}$-module which is stable under the monoid or comonoid structure maps. A weight-graded cooperad or operad is \emph{connected} if the unit or counit is an isomorphism on weight $0$.
A connected weight-graded operad has a unique weight-respecting augmentation. 
A connected weight-graded cooperad has a unique weight-respecting coaugmentation. 
A \emph{reduced} $\mathbb{N}$-module or $\mathbb{S}$-module has $0$ in arity $0$.

We will always work with the categories of \emph{reduced connected weight-graded} operads and \emph{reduced connected weight-graded} cooperads, equipped with their unique weight-respecting (co)augmentations. 
We will henceforth suppress these adjectives in our terminology.

\begin{remark}
\label{remark: generalizations}
Essentially everything we do works equally well in the symmetric and non-symmetric case, so we use the same symbol for both products; the reader should interpret it as appropriate for the context.

The ``reduced'' assumption (starting our indexing at $1$ and not $0$) is a common technical restriction to make certain sums finite. This assumption is probably extraneous in the context of a connected weight-grading but various technical results on which we rely are stated in the reduced category and it would require careful verification that no subtle problems arise.

The connected weight-grading assumption is only used in Proposition~\ref{proposition: formal implies classically formal, flat NS or char 0 field} and Theorem~\ref{theorem:HPL}. 
This assumption would not be necessary in any situation in which the conclusions of these two results were known by other means.
For example, model-categorical methods have often been used to give access to more powerful but non-constructive existence statements in operad theory. 
So one kind of context that might work would use model-theoretical tools to give these two statements, probably combining a model category of operads and a model category of algebras over (sufficiently nice) operads. 
This would require that the interface with the classical theory be well-developed enough to give statements about the specific concrete model for homotopy algebras in use in these two theorems.

In any event, the examples that arise most commonly are reduced and support a connected weight-grading.

It may also be possible to weaken the characteristic zero assumption for $\mathbb{S}$-modules. 
This would be much more interesting in the sense that algebras over symmetric operads arise often in characteristic $p$. 
However, the requisite changes appear to be much more substantial and it is not entirely clear that everything would work.
\end{remark}

\subsection{Algebras and coalgebras}
\label{subsec: algebras and coalgebras}
There is a fully faithful functor from the category of chain complexes over $\groundring$ to the category of either $\mathbb{N}$-modules or $\mathbb{S}$-modules which takes a chain complex $V$ to the object with $V$ in arity $0$ and the zero complex in all other arities.
In a slight abuse, we will write $V$ for the image of $V$ under this functor as well, hoping it causes little confusion.
For such an object $V$ we write $Q(V)$ as shorthand for the $\mathbb{N}$-module or $\mathbb{S}$-module $Q\compositionproduct V$. 
This is also necessarily concentrated in arity zero.

An \emph{algebra} over the operad $P$ is a left $P$-module whose underlying $\mathbb{S}$-module is concentrated in arity $0$. 
That is, it is a chain complex $V$ equipped with a map $P(V)\to V$ which satisfies the usual associativity and unitality constraints.
Similarly, a \emph{coalgebra} over the cooperad $C$ is a left $C$-comodule whose underlying $\mathbb{S}$-module is concentrated in arity $0$. That is, it is a chain complex $V$ equipped with a map $V\to C(V)$ which satisfies the usual coassociativity and counitality constraints. 
We do not require any compatibility with the weight-grading in either case.

Fix a cooperad $C$ with structure map $\Delta$ and counit $\epsilon$. 
The cofree conilpotent $C$-coalgebra on $V$ is the coalgebra $C(V)$ with structure map induced by $\Delta$. 
By abuse, we typically also use the notation $\Delta$ for this structure map 
\[
C(V)\xrightarrow{\Delta}(C\compositionproduct C)(V)\cong C(C(V)).
\] 
There will be several further times when we abuse notation like this, using the symbol of a map $f$ for the map obtained by taking the monoidal product of $f$ with some identity map or other.

By the universal property of being cofree, given a chain map $f:C(V)\to W$, there is a unique extension to a map $F$ of $C$-coalgebras $C(V)\to C(W)$. 
Explicitly, $F$ is given by the composite
\[
C(V)\xrightarrow{\Delta} C(C(V))\xrightarrow{f}C(W)
\]
where $\Delta$ is the comonoidal structure map of $C$.

\subsection{Coderivations}
\label{subsec: coderivation}
Given a chain map $m:C(V)\to V$, there is another useful extension of $m$, this time not as a coalgebra map but as a coderivation.
To discuss this, we first recall the linearization of the monoidal product $\compositionproduct$ (see, e.g.,~\cite[6.1]{LodayVallette:AO} for more details). 
We describe the non-symmetric case for ease of notation.

The linearization of $\compositionproduct$ is easiest to describe in terms of a trifunctor on $\mathbb{N}$-modules. 
Given three $\mathbb{N}$-modules $P$, $Q$, and $R$, the product $P\compositionproduct(Q;R)$ consists of the $R$-linear summands of $P\compositionproduct (Q\oplus R)$, i.e., 
\begin{align*}
(P\compositionproduct(Q;R))(n)&=\bigoplus_{k=1}^\infty\bigoplus_{i=1}^k\bigoplus_{n_1+\cdots+n_k=n} P(k)\otimes Q(n_1)\otimes \cdots\otimes R(n_i)\otimes \cdots\otimes Q(n_k). 
\end{align*}
Again we abbreviate $P\compositionproduct (Q; R)$ as $P(Q;R)$ if both $Q$ and $R$ are in the image of the inclusion from chain complexes to $\mathbb{N}$-modules.
It doesn't make sense to ask about ``associativity'' but this trifunctor satisfies the following compatibility relations:
\begin{align}
\label{eq:non-associative compatibility of trifunctor 1}
(P\compositionproduct Q)\compositionproduct (R;S) &\cong P\compositionproduct(Q\compositionproduct R; Q\compositionproduct(R;S))
\\
\label{eq:non-associative compatibility of trifunctor 2}
(P\compositionproduct (Q;R))\compositionproduct S &\cong P\compositionproduct(Q\compositionproduct S;R\compositionproduct S).
\end{align}

Now we can define the linearization of $\compositionproduct$ as $P\compositionproduct_{(1)}R=P\compositionproduct(I;R)$. 
Explicitly we have
\begin{align*}
(P\compositionproduct_{(1)}R)(n)&=
(P\compositionproduct(I;R))(n)
\\
&=\bigoplus_{k=1}^\infty\bigoplus_{i=1}^k P(k)\otimes I(1)\otimes\cdots\otimes R(n-k+1)\otimes \cdots\otimes I(1) 
\\
&\cong\bigoplus_{k=1}^{n+1}\bigoplus_{i=1}^k  P(k)\otimes R(n-k+1).
\end{align*}
The linearization is not associative but rather preLie in general. We will not need the preLie compatibility but rather two other relations which follow from Equations~\eqref{eq:non-associative compatibility of trifunctor 1} and~\eqref{eq:non-associative compatibility of trifunctor 2}: 
\begin{align}
\label{eq:non-associative compatibility of infinitesimal composition 1}
(P\compositionproduct Q)\compositionproduct_{(1)} R 
&\cong P\compositionproduct(Q;Q\compositionproduct_{(1)}R)
\\
\label{eq:non-associative compatibility of infinitesimal composition 2}
(P\compositionproduct_{(1)} Q)\compositionproduct R &\cong P\compositionproduct(R;Q\compositionproduct R).
\end{align}
These isomorphisms and the associator isomorphism for $\compositionproduct$ together satisfy the appropriate analogues of the pentagon relation, and so we may safely suppress them, assuming a unique natural isomorphism between any two parenthesizations that are equivalent by a chain of (modified) such associators.

The linearization $P\compositionproduct_{(1)}R$ is a direct summand of $P\compositionproduct R$ with complement given in terms of similar formulas with either zero or more than one entry from $R$.
A map $O\to Q\oplus R$ induces a map $P\compositionproduct O\to P\compositionproduct (Q;R)$, and likewise a map $Q\oplus R\to O$ induces a map $P\compositionproduct(Q;R)\to P\compositionproduct O$.

Now returning to our cooperad $C$, there is a linearized coproduct from the map $C\xrightarrow{\epsilon,\id_C}I\oplus C$ as follows:
\[
C\xrightarrow{\Delta}C\compositionproduct C\xrightarrow{} C\compositionproduct_{(1)} C
\]
which we denote $\Delta_{(1)}$.

Now given a differential graded $C$-coalgebra $X$, a linear map $M:X\to X$ of homological degree $-1$ is a \emph{coderivation} of $X$ if the composition along the top and right side of the following diagram is equal to the sum of the composition along the left side and the two choices on the bottom:
\[
\begin{tikzcd}
X\dar[swap]{\Delta}\ar[rrrr,"M"]
&&&&
X\dar{\Delta}
\\
C(X)\rar\ar[rrrr,swap, bend right,"d_C(X)"]
&C(X;X)
\ar[rr,swap,"C(X;M)"]
&&
C(X;X)
\rar
&
C(X)
\end{tikzcd}
\]
where the unmarked arrows are induced by the diagonal $X\to X\oplus X$ and the fold map $X\oplus X\to X$. 
A priori, this definition does not use the weight-grading of the cooperad $C$. 

Given a chain complex $V$, the coderivations of the free $C$-coalgebra $C(V)$ are in bijection with the homological degree $-1$ linear maps $C(V)\to V$. 
Given a coderivation $M$, one gets a map $m:C(V)\to V$ by projecting to the cogenerators:
\[
C(V)\xrightarrow{M} C(V)\to V.
\]
In the other direction, given a map $m:C(V)\to V$, one obtains a linear map $C(V)\to C(V)$ via adding $d_C(V)$ to the composition
\[
C(V)\xrightarrow{\Delta_{(1)}}(C\compositionproduct_{(1)}C)(V)\cong C(V;C(V))\xrightarrow{m} C(V;V)\to C(V).
\]
Here the isomorphism is via Equation~\eqref{eq:non-associative compatibility of infinitesimal composition 1} and the unmarked arrow is induced by the fold map of $V$.
It is a tedious diagram chase to verify that this indeed gives a coderivation. 

There is a similar linearization $\compositionproduct_{(1)}$ in the symmetric case which we will not write down explicitly and all of the statements in this section work symmetrically with the minimal requisite changes.

\subsection{Homotopy algebras}

One convenient method for describing homotopy algebras in the operadic formalism is via conilpotent cooperads and the cobar functor. For our purposes here, conilpotence can essentially remain a black box, but we briefly outline the definition with references for the interested reader.

Given a coaugmentation of a cooperad $C$, one can define a \emph{coradical filtration} of $C$ of the form
\[I=F_0C\subset F_1C\subset \cdots \subset C\]
where the object $F_1C$ is the coradical of $C$ with respect to the coaugmentation. 
There are different definitions of $F_nC$ for $n>1$ in the literature (see, e.g.,~\cite[\S~5.8.5]{LodayVallette:AO} or~\cite[Definition~4.3]{LegrignouLejay:HTLC}) which try to capture the idea that $F_nC$ admits at most $n-1$ ``interesting'' applications of the structure map $\Delta$ before becoming merely formal extensions by $I$. 
Then one says that the coaugmented cooperad $C$ is conilpotent (called \emph{local conilpotent} in~\cite{LegrignouLejay:HTLC}) if the natural map 
\[\colim_i F_i C \to C\] 
is an isomorphism.
Despite the different filtrations in use in the definitions, a cooperad is conilpotent in the sense of~\cite{LodayVallette:AO} if and only if it is locally conilpotent in the sense of~\cite{LegrignouLejay:HTLC}, and for parsimony we will use ``conilpotent'' to refer to the cooperads satisfying these equivalent conditions.

The \emph{coaugmentation coideal} $\overline{C}$ of a coaugmented cooperad is the linear cokernel of the coaugmentation $I\to C$. 
Then the \emph{cobar functor} $\Omega$ from conilpotent cooperads to operads takes the cooperad $C$ to the free operad on $\bar{C}[1]$, equipped with 
\begin{enumerate} \item a differential that combines the internal differential of $C$ and the comonoid structure of $C$ and 
\item the induced weight-grading.
\end{enumerate} See, e.g.,~\cite[\S~6.5.2]{LodayVallette:AO}. 
This works equally well for nonsymmetric cooperads and operads. 
A \emph{twisting morphism} from a conilpotent cooperad $C$ to an operad $P$ is a morphism of operads from $\Omega(C)$ to $P$; the twisting morphism is \emph{Koszul} if it induces an isomorphism on homology groups. Since $\Omega(C)$ is a free operad, a twisting morphism is entirely determined by its restriction to $\overline{C}[1]$. So we could alternatively define, as in \cite[\S 6.4]{LodayVallette:AO}, a twisting morphism as a degree $-1$ map $C\to P$ satisfying certain equations. By definition our twisting morphisms must intertwine the weight-grading of $C$ and $P$.

\begin{assumption}
\label{assumption: setup}
Let $P$ be an operad over $\groundring$, concentrated in degree zero. Let $C$ be a conilpotent cooperad over $\groundring$, with coaugmentation coideal $\bar{C}$ concentrated in strictly positive degree. Let $\kappa:\Omega(C)\to P$ be a Koszul twisting morphism.
\end{assumption}

Given Assumption~\ref{assumption: setup}, one model for the category of \emph{strongly homotopy $P$-algebras}, or \emph{$P_\infty$-algebras} is the category of cofree conilpotent $C$-coalgebras.
\begin{defi}
\label{defi:infty morphism}
Assume Assumption~\ref{assumption: setup}.
A \emph{$P_\infty$-algebra} structure $M$ on a chain complex $(V,d)$ is a degree $-1$ square-zero coderivation of the cofree conilpotent coalgebra $C(V)$ so that the composition
\[
V\xrightarrow{\text{coaugmentation}} C(V)\xrightarrow{M} C(V)\xrightarrow{\text{projection}} V
\]
is the differential $d$.

A $P_\infty$-algebra structure on a graded $R$-module $V$ is a $P_\infty$-algebra structure on $(V,d)$ for some differential $d$.

A \emph{$P_\infty$-morphism} between $(V,d,M)$ and $(V',d',M')$ is a map $F$ of differential graded $C$-coalgebras from $(C(V),M)$ to $(C(V'),M')$, i.e., a map of coalgebras so that the following diagram commutes:
\[
\begin{tikzcd}
C(V)\rar{F}\dar{M} & C(V')\dar{M'}\\
C(V)\rar{F} & C(V').
\end{tikzcd}
\]
A $P_\infty$ morphism is a \emph{quasi-isomorphism} if the composition 
\[
V\xrightarrow{\text{coaugmentation}} C(V)\xrightarrow{F} C(V')\xrightarrow{\text{projection}} V'
\]
is a quasi-isomorphism of chain complexes.
\end{defi}
\begin{remark}
\label{remark: connection between P and P-infinity algebras}
As given here, this definition does not match the terminology for an algebra over an operad from Section~\ref{subsec: algebras and coalgebras}.

There is an alternate characterization of $P_\infty$-algebras that brings the two usages into closer but not perfect alignment. Namely, by~\cite[Theorem 10.1.13]{LodayVallette:AO}, a $P_\infty$-algebra structure on $(V,d)$ is equivalent to an $(\Omega C)$-algebra structure on $(V,d)$ in our earlier sense. 
Every $(\Omega C)$-algebra morphism is a $P_\infty$-algebra morphism in this new sense~\cite[Proposition~10.2.5]{LodayVallette:AO}, but the converse is not true in general.

From this alternate perspective, any $P$-algebra structure on $(V,d)$ can be pulled back along the Koszul twisting morphism $\kappa:\Omega C\to P$ to a $P_\infty$-algebra structure on $(V,d)$. 
Similarly, any $P$-algebra morphism $(V,d,m)\to (V',d',m')$ pulls back along $\kappa$ to a $P_\infty$ morphism.
This pullback constitutes a functor from the category of $P$-algebras to the category of $P_\infty$-algebras. 
If $\kappa$ is surjective, this functor is in fact the inclusion of a (non-full) subcategory.

This perspective also makes it easier to see that quasi-isomorphisms are closed under composition.
\end{remark}

By the discussion of the last section, specifying the coderivation $M$ is equivalent to specifying a degree $-1$ linear map
\[
C(V)\xrightarrow{m} V
\]
so that $V\to C(V)\to V$ is $d$ (the map $m$ must satisfy further conditions equivalent to the equation $M^2=0$).

We will consistently pass back and forth between these two representations of the same data in this article by using a capital letter for the coderivation and the corresponding lowercase letter for the projection to the cogenerators.
E.g., $(V,M)$ and $(V,m)$ are both notation for the same $P_\infty$-algebra structure on $V$ and $M:C(V)\to C(V)$ is the $C$-coalgebra coderivation extending $m:C(V)\to V$.

Similarly, we encode a $P_\infty$-morphism from $V$ to $W$ (leaving the structures implicit) via a linear map
$C(V)\to W$ satisfying relations which imply that the corresponding coalgebra morphism intertwines the coderivations.
Again, we use capital letters for the coalgebra maps and lower case letters for the projections to the cogenerators.

\begin{terminology}
\label{terminology: components}
Given a map (say, $m$ or $f$) from $C(V)$ to $W$, we use a subscript to indicate the further decomposition with respect to the homological degree of $C$, and call the resulting maps \emph{components}. Explicitly, we can decompose $C(V)$ as
\[C(V)=\bigoplus_{i\in\mathbb{N}}C_i(V)\]
where $C_i(V)$ is defined as
\[
C_i(V)=\bigoplus_{k\in\mathbb{N}}C_i(k)\otimes_{\mathbb{S}_k} V^{\otimes k}
\qquad{}\text{or}\qquad
C_i(V)=\bigoplus_{k\in\mathbb{N}}C_i(k)\otimes V^{\otimes k}
\]
respectively in the symmetric or non-symmetric contexts. 
We denote by $f_i$ or $m_i$, the restriction of $f$ or $m$ to $C_i(V)$.
\end{terminology} 
So for instance $m_i:C_i(V)\to V$ is the \emph{$i$th component} of the $P_\infty$-algebra $m$ and $f_i:C_i(V)\to W$ is the \emph{$i$th component} of a $P_\infty$-morphism $f$.
\begin{remark}[Warning]
\label{remark: warning}
This does not in general coincide with other uses of the term \emph{component} in operadic algebra. 
Often the $i$th component would be the component in \emph{arity} $i$ in $C$.
A different usage common in the literature would have the $i$th component refer to the component in \emph{weight-grading} $i$ in $C$.
For us it is neither of these but the component in \emph{homological degree} $i$ in $C$. 

In any case, typically the desirable properties for components are that they be $\mathbb{N}$-indexed and that the $0$th component be split by the coaugmentation and counit, which occurs for us by Assumption~\ref{assumption: setup}. 

In the case of weight-graded algebras, the arguments given in this paper should work with more general filtrations mutatis mutandis (Definition~\ref{defi: degree twisting} would have to be changed to be about the weight rather than the homological degree).
\end{remark}
\begin{remark}
Note that our usage does not coincide naively with the homological degree of the operations. 
Since $M$ is supposed to be of degree $-1$, the operations of form $m_i$ have degree $i-1$.
For maps $F$ (of degree zero) the $f_i$ operations indeed have degree $i$.
\end{remark}

\begin{remark}
Given $P$ concentrated in degree zero, the image $BP$ of $P$ under the bar functor satisfies the conditions of Assumption~\ref{assumption: setup}: it is conilpotent, the coaugmentation coideal is concentrated in strictly positive degrees, and it comes with a canonical Koszul twisting morphism to $P$. 
So if we don't care about the details of the particular choice of model for the category of strongly homotopy $P$-algebras, we need only begin with an operad over $\groundring$ concentrated in degree zero.

However, many practitioners have preferred models, especially in specific cases. 
A classical example where there is a different preferred model arises in the context of a so-called \emph{Koszul operad}, which is equipped with a quadratic presentation which yields a Koszul twisting morphism $\Omega C\to P$ from a cooperad $C$ without an internal differential (see~\cite[7.4, especially Thm.~7.4.2(4)]{LodayVallette:AO}, although this case goes back to~\cite{GinzburgKapranov:KDO}).
Then the model which arises for the category of strongly homotopy $P$-algebras is the so-called \emph{minimal model}~\cite[6.3.4, Corollary 7.4.3]{LodayVallette:AO}.

Since sometimes such a preference exists, we have explicitly recorded the required conditions on the governing cooperad $C$ and its relation to $P$.
\end{remark}

\begin{example}
The classical examples (the ``three graces'') and other Koszul operads are all examples. In particular, the following examples work.
\begin{enumerate}
	\item Let $P$ be the associative operad and $C$ the shifted coassociative cooperad. Then $P_\infty$-algebras are $A_\infty$-algebras with the standard definitions and the example fits into this framework.
	\item Let $P$ be the Lie operad and $C$ the shifted cocommutative cooperad. Then $P_\infty$-algebras are $L_\infty$-algebras with the standard definitions and the example fits into this framework.
	\item Let $P$ be the commutative operad and $C$ the shifted coLie cooperad. Then $P_\infty$-algebras are $C_\infty$-algebras with the standard definitions and the example fits into this framework.
	\item All of these are subsumed by the following. Let $P$ be a Koszul operad concentrated in degree zero. Then $P_\infty$-algebras with the standard definitions fit into this framework.
\end{enumerate}
\end{example}
\begin{example}
Everything we will do works with only the evident requisite changes for colored operads.
So another example that works is for $P$ the colored operad whose algebras are non-unital non-symmetric Markl operads. 
The operad $P$ is concentrated in degree zero so it fits into this framework. The operad $P$ is also Koszul with the requisite changes for color sets so there is an explicit small model for homotopy operads as $P_\infty$-algebras~\cite{VanderLaan:CKDSHO}. 
\end{example}

We will use the following elementary observation about a situation relating a map and its linearization.
\begin{lemma}
\label{lemma: technical compatibility between f and linearization when components vanish.}
Let $C$ be a cooperad and $V$ a chain complex, and suppose given a map $f:C(V)\to V$. Suppose further that $f_0=\id_V$ and $f_i=0$ for $1\le i<n$.
Then the projection
\[
C_N(V)\xrightarrow{\Delta}C(C(V))\xrightarrow{f}C(V)\to C_{N-n}(V)
\]
and the projection
\[
C_N(V)\xrightarrow{\Delta_{(1)}}C(V; C(V))\xrightarrow{f_n}C(V;V)\to C_{N-n}(V)
\]
coincide for $n>0$.
\end{lemma}
\begin{proof}
Applying $\Delta$ lands in a sum over partitions of $n=i_1+\cdots + i_k$ into summands represented by tensor products of the form
\[
C_{N-n}(k)\otimes C_{i_1}\otimes\cdots \otimes C_{i_k}
\]
and $f$ vanishes on all summands except those with a single index of value $n$ and all other indices of value $0$. 
The sum of such terms is then the projection of $C\compositionproduct C$ to $C\compositionproduct(C_0;C_n)$.
\end{proof}

\subsection{Homological perturbation and formality}
Given a differential graded $P$-algebra $(A,d,m)$, there is an induced $P$-algebra structure on the homology $H(A,d)$, because the operations making up $m$ are all chain maps.
However, in general the induced structure $(H(A,d),0,m_*)$ is not equivalent to the original $P$-algebra $(A,d,m)$.
\begin{defi}
We call the differential graded $P$-algebra $(A,d,m)$ \emph{classically formal} if it is equivalent to the induced structure $(H(A,d),0,m_*)$, i.e., if there exists a differential graded $P$-algebra $(\hat{A},\hat{d},\hat{m})$ and a zig-zag of maps of differential graded $P$-algebras inducing isomorphisms on homology:
\[
(A,d,m)\xleftarrow{\sim} (\hat{A},\hat{d},\hat{m})\xrightarrow{\sim} (H(A,d),0,m_*).
\]
\end{defi}
Another way to say this is that $(A,d,m)$ and $(H(A,d),0,m_*)$ are isomorphic objects in the homotopy category of differential graded $P$-algebras. 
This notion goes back to~\cite{DeligneGriffithsMorganSullivan:RHTKM}. Let us point out that, usually, such an algebra is simply called formal. We have decided to change the terminology in this paper because we would like to operate under a slightly different definition of formality.

\begin{defi}
We call the differential graded $P$-algebra $(A,d,m)$ \emph{formal} if there is a $P_\infty$-quasi-isomorphism from $(A,d,m)$ to $(H(A,d),0,m_*)$.
\end{defi}

In the nicest cases, the two notions of formality coincide. Here are two precise statements along these lines.
\begin{prop}
\label{proposition: classically formal implies formal, field}
Let $\groundring$ be a field. A classically formal differential graded $P$-algebra is formal.
\end{prop}
This is more or less well-known, we include a proof later as a corollary of Theorem~\ref{theorem:HPL}.
\begin{prop}
\label{proposition: formal implies classically formal, flat NS or char 0 field}
Let $\groundring$ be a characteristic zero field. Then a formal differential graded $P$-algebra is classically formal.

Let $\groundring$ be any commutative ring. Let $P$ and $C$ be non-symmetric, and aritywise and degreewise flat. Then a formal differential graded $P$-algebra is classically formal.
\end{prop}
\begin{proof}
The first statement can be found in~\cite[(M1),(M3)]{Markl:HAHA} and in \cite[Theorem 11.4.9]{LodayVallette:AO}.

We were not able to find a reference for the second statement so we give a few details. 

We first recall the bar and cobar construction associated to the twisting morphism $\kappa$, denoted $B_\kappa$ and $\Omega_\kappa$. The functor $B_\kappa$ takes as an argument a $P$-algebra and returns a $C$-coalgebra, the functor $\Omega_\kappa$ is its left adjoint. The functor $B_\kappa$ can be extended to $P_\infty$-algebras. In that case, it is denoted $B_\iota$ as it is the Bar construction associated to a canonical twisting morphism $\iota$ (see \cite[\S 11.4]{LodayVallette:AO}) which in our presentation corresponds to the identity map $\Omega C\to \Omega C$. Note that if we view $P_\infty$-algebras as a subcategory of differential graded $C$-coalgebras, the functor $B_\iota$ is simply the inclusion functor. 

Now, we sketch the proof of the theorem. Let $A=(V,d,m)$ be a $P$-algebra and $H=(H(A,d),0,m_*)$ its homology with the induced $P$-algebra structure.

The $C$-coalgebras $B_\iota(A)$ and $B_\iota(H)$ are simply $B_\kappa(A)$ and $B_\kappa(H)$ since $A$ and $H$ are strict $P$-algebras. Therefore by \cite[Theorem 11.3.3]{LodayVallette:AO}, the counit of the adjunction $(\Omega_\kappa,B_\kappa)$ induces quasi-isomorphisms of $P$-algebras $\Omega_\kappa B_\kappa(A)\to A$ and $\Omega_\kappa B_\kappa(H)\to H$.

By assumption, we have a $P_\infty$-morphism $f:A\to H$ inducing an isomorphism in homology. By \cite[Proposition 11.4.7]{LodayVallette:AO}, this implies that the map of $P$-algebras
\[
\Omega_\kappa B_\kappa(A)\cong\Omega_\kappa B_\iota(A)\xrightarrow{\Omega_\kappa B_\iota(f)} \Omega_\kappa B_\iota (H)
\cong \Omega_\kappa B_\kappa(H)
\]
is a quasi-isomorphism.
Combining these three maps yields a zig-zag of quasi-isomorphisms of $P$-algebras
\[A\xleftarrow{\simeq}\Omega_\kappa B_\kappa(A)\xrightarrow{\simeq}\Omega_\kappa B_\kappa(H)\xrightarrow{\simeq} H.\]

We invite the skeptical reader to check that the only reason the ground ring is assumed to be a field of characteristic zero in this context in~\cite{LodayVallette:AO} is so that the operadic K\"unneth formula \cite[Proposition 6.2.3]{LodayVallette:AO} holds. 
Since we assume that $P$ and $C$ are flat, we have a version of this proposition in the non-symmetric case. 
We should be a little careful because a priori we could have terms like $H(A^{\otimes n},d_{A^{\otimes n}})$ on the right side of the formula which would necessitate flatness of $A$ and/or $H$ to continue.
But the only such terms that arise in the proof of Theorem 11.3.3 have $n=1$, so further assumptions are unnecessary.
\end{proof}

Even in the absence of formality and working over a general ground ring, as long as $(A,d)$ and $(H(A,d),0)$ are homotopy equivalent chain complexes, there is always a way to compress the data of $(A,d,m)$ to the homology, via the so-called homological perturbation lemma or transfer theorem. 
Perturbation methods are a classical tool in algebraic topology. 
They were first used for $A_\infty$-algebras by Kadeishvili~\cite{Kadeishvili:OTHFS}. 
For algebras over more general operads, see~\cite{Markl:HAHA,LodayVallette:AO,Berglund:HPTAO}. 

\begin{theorem}[Transfer theorem]
\label{theorem:HPL}
Let $(A,d,m)$ be a differential graded $P$-algebra such that the chain complex $(H(A,d),0)$ can be written as a homotopy retract of $(A,d)$. Then there exist
\begin{enumerate}
\item a transferred $P_\infty$-algebra structure $m^t$ with zero differential on $H(A,d)$ extending the induced $P$-algebra structure on the homology and
\item quasi-inverse $P_\infty$ quasi-isomorphisms between the $P_\infty$-algebras $(A,d,m)$ and $(H(A,d),0,m^t)$ extending a given homotopy retraction between $(A,d)$ and $(H(A,d),0)$.
\end{enumerate}
\end{theorem}
Here quasi-inverse means that the composition in both directions induces the identity morphism on homology.

Note that such a homotopy retraction always exists over a field but this also holds if both $A$ and $H(A,d)$ are degreewise projective and the base ring $\groundring$ is hereditary. Indeed if this is the case, we have an epimorphism $Z_n(A)\to H_n(A)$ from the group of $n$-cycles to the $n$-th homology group. This epimorphism splits since $H_n(A)$ is projective. We can thus write $Z_n(A)$ as the direct sum $H_n(A)\oplus B_n(A)$. Moreover, we can identify $A_n/Z_n(A)$ with $B_{n-1}(A)$ which is projective. Therefore the epimorphism $A_n\to B_{n-1}(A)$ induced by the differential also splits and we have a splitting $A_n\cong B_n(A)\oplus H_n(A)\oplus B_{n-1}(A)$. The construction of the homotopy retraction then works exactly as in the case of fields.
\begin{proof}[Proof of Proposition~\ref{proposition: classically formal implies formal, field}]
It suffices to show that given a quasi-isomorphism $f$ of $P$-algebras from $(\hat{A},\hat{d},\hat{m})$ to $(A,d,m)$, there is a $P_\infty$-quasi-isomorphism in the other direction $(A,d,m)\to (\hat{A},\hat{d},\hat{m})$.

Because $\groundring$ is a field, the hypotheses of Theorem~\ref{theorem:HPL} apply to both the domain and codomain of $f$.
Precomposing and postcomposing $f$ with the guaranteed $P_\infty$-quasi-isomorphisms yields a $P_\infty$-quasi-isomorphism from $(H(\hat{A},\hat{d}),0,\hat{m}^t)$ to $(H(A,d),0,m^t)$.
Because the differentials on both sides of this map vanish, this $P_\infty$-quasi-isomorphism is an isomorphism of $P_\infty$ algebras and admits an inverse (see~\cite[Theorem~10.4.1]{LodayVallette:AO}). 
The inverse can be precomposed and postcomposed with the guaranteed $P_\infty$-quasi-isomorphisms to finally yield a $P_\infty$ quasi-isomorphism from $(A,d,m)$ to $(\hat{A},\hat{d},\hat{m})$.
\end{proof}

\begin{remark}
The version of Theorem~\ref{theorem:HPL} stated in~\cite{Berglund:HPTAO}, which works for arbitrary rings makes the assumption that $P(1)\cong I$. In the subcategory of operads and cooperads with $0$ in arity $0$ and $I$ in arity $1$, there there is a canonical connected weight-grading on every operad and cooperad by one less than the arity, and Berglund's methods apply mutatis mutandis to the more general connected weight-graded case.
\end{remark}

Until the end of this section, we let $(A,d,m)$ be a differential graded $P$-algebra such that the chain complex $(H(A,d),0)$ can be written as a homotopy retract of $(A,d)$ so that the transfer theorem applies. 

\begin{cor}
If $(H(A,d),0,m^t)$ and $(H(A,d),0,m_*)$ are isomorphic as $P_\infty$-algebras, then the algebra $(A,d,m)$ is formal.
\end{cor}

\begin{defi}
\label{defi: n-formality}
Let $n$ be a positive integer, we say that a differential graded $P$-algebra $(A,d,m)$ is \emph{$n$-formal} if $(H(A,d),0,m^t)$ is isomorphic to a $P_\infty$-algebra $(H(A,d),0,m)$ with $m_i=0$ for $i$ in the range $2\leq i\leq n$.
\end{defi}

\begin{remark}[Warning]
This notion differs from other notions of $n$-formality in the literature for commutative differential graded algebras~\cite[Definition~2.2]{FernandezMunoz:FDS},~\cite[Definition~2.4]{Macinic:CRFPNG}, and~\cite[Introduction]{CiriciHorel:ECPFTC}. 
The definitions in these references use $n$ as a stand-in for geometric dimension, so that $n$-formal means something along the lines of ``formal up to homological degree $n$'' whereas for us an $n$-formal commutative differential graded algebra is ``formal up to arity $n+2$.'' 




\end{remark}

The following two propositions that show that $n$-formality can sometimes imply formality.

\begin{prop}
\label{prop: chains n-formality implies formality}
Let $(A,d,m)$ be a differential graded $P$-algebra. Assume that $H_i(A,d)=0$ for $i$ outside of the interval $[0,n]$. Then, $(A,d,m)$ is $n$-formal if and only if it is formal. 
\end{prop}

\begin{proof}
Indeed, the components $m_i$ with $i>n$ of any $P_\infty$-structure on $H(A,d)$ have to be zero for degree reasons.
\end{proof}

There is also a somewhat more involved version for cohomologically graded algebraic structures. 
In order to keep consistent conventions throughout, we phrase it in terms of homologically graded structures concentrated in nonpositive degrees; statements with nonnegative cohomological grading can be obtained by negating indices.
\begin{prop}
\label{prop: cochains n-formality implies formality}
Let $(A,d,m)$ be a differential graded $P$-algebra. Let $j$ be an integer, $n$ and $q$ a positive integers. Suppose that
\begin{enumerate}
\item for all $i$, the component $C_i$ is concentrated in arity at least $i+j$, and 
\item $H_i(A,d)=0$ for $i$ outside $[n-q(n+j+1)+2,-q]$.
\end{enumerate}
Then $(A,d,m)$ is $n$-formal if and only if it is formal.
\end{prop}
The most common applications occur in the case where 
\begin{itemize}
\item The operad $P$ is trivial in arity $0$ and $1$, in which case $C$ can be chosen to satisfy (1) with $j = 2$, and
\item The integer $q$ is specified to be $2$ (the \emph{simply connected} case)
\end{itemize}
in which case the restriction is that $H_i(A,d)=0$ for $i$ outside $[-n-4,-2]$.
\begin{proof}
Applying $m_i$ to elements in degree at most $-q$ yields something in degree at most $i-(i+j)q$.
If $i\ge n+1$ then 
\[
i-(i+j)q =  i(1-q)-jq \le n-nq-jq-q+1
\]
so the output is outside of the support of $H(A,d)$.
\end{proof}

\section{Extending formality inductively}
\label{section: induction}
Throughout this section we will assume given the data $(P,C,\kappa)$ of Assumption~\ref{assumption: setup}.
\subsection{Statement of the key lemma}
The following lemma is the technical core of the argument that we will use.
\begin{lemma}
\label{lemma: key lemma}
Let $V$ be a graded $\groundring$-module and $\sigma_{\alpha}:V\to V$ be the degree twisting morphism by $\alpha$ where $\alpha$ is a unit in $\groundring$ which is such that $\alpha^n-1$ is also a unit. Suppose that $V$ (viewed as a chain complex with trivial differential) is equipped with
\begin{enumerate}
\item a $P_\infty$-algebra structure $C(V)\xrightarrow{m} V$ which vanishes on $C_{i}(V)$ for $i=0$ and $i$ in the range $2\le i< n+1$ and
\item a $P_\infty$-automorphism $s$ of $(V,m)$ which is equal to the automorphism $\sigma_{\alpha}$ on $V\cong C_{0}(V)\to V$ and which vanishes on $C_{i}(V)$ in the range $1\le i< n$.
\end{enumerate}
Then there exist:
\begin{enumerate}
\item a $P_\infty$-algebra structure $m'$ on $V$ which vanishes on $C_{i}(V)$ for $i=0$ and $i$ in the range $2\le i< n+2$,
\item a $P_\infty$-automorphism $s'$ of $(V,m')$ which is equal to the automorphism $\sigma_{\alpha}$ of $V$ on $V\cong C_{0}(V)\to V$ and which vanishes on $C_{i}(V)$ in the range $1\le i<n+1$, and
\item a $P_\infty$-isomorphism $f$ from $(V,m)$ to $(V,m')$ which is equal to $\id_V$ on $V\cong C_{0}(V)\to V$,  vanishes on $C_{i}(V)$ for $i\notin \{0,n\}$, and intertwines the $P_\infty$-automorphisms $s$ and $s'$.
\end{enumerate}
\end{lemma}
In words, given $m$ which is formal up to degree $n$ components and $s$ which is formal up to degree $n-1$ components, we can build isomorphic data $m'$ and $s'$ with the formality range improved by $1$.
\subsection{A warmup example}
The proof of Lemma~\ref{lemma: key lemma}, given in the next subsection, is somewhat technical. 
Here, for the reader's convenience, we specialize to a familiar case, with $P$ be the associative operad and $C$ the shifted coassociative cooperad over $\mathbb{Q}$, so that $P_\infty$-algebras are standard $A_\infty$ algebras over $\mathbb{Q}$. 
We will explicate the first instance of Lemma~\ref{lemma: key lemma} in this case.
This subsection is purely expositional and none of the rest of the article has any formal reliance on it. 

\begin{remark}[Warning]
In this example we will also deviate from the notation we have set up to relate this more directly to the typical notation used in the literature on $A_\infty$ algebras.
\end{remark}
\begin{remark}
None of the arguments are particularly sensitive to sign conventions.
The purpose of this example is to illustrate the induction, not to give a precise proof (that is the purpose of the following subsection).
Therefore, in the example we will be somewhat careless with signs.
\end{remark}
Here, using one typical set of degree conventions for $A_\infty$ algebras, we will consider an $A_\infty$ algebra $(V,d,\overline{\mu}) = (V,d,\mu_2,\mu_3,\ldots)$ as a graded vector space $V$ equipped with a degree $-1$ differential $d$ and a collection of maps $\mu_i$ of degree $i-2$ from $V^{\otimes i}$ to $V$ satisfying the $A_\infty$ relations.
With these conventions, an automorphism $\overline{s}$ of $(V,d,\overline{\mu})$ is a collection $(s_1,s_2,\ldots)$ of maps where $s_i$ is a degree $i-1$ map from $V^{\otimes i}$ to $V$ satisfying the appropriate compatibility relations with $\mu_i$ and $d$.

\begin{example}[the associative $n=1$ case]
We are given 
\begin{itemize}
\item an $A_\infty$ algebra $(V,d,\overline{\mu})$,
\item an $A_\infty$ automorphism $\overline{s}$ of $(V,d,\overline{\mu})$, and
\item a rational number $\alpha$
\end{itemize} 
such that:
\begin{itemize}
\item the differential $d$ of $V$ vanishes (this is the condition that $m$ vanishes on $C_0(V)$),
\item the arity one component $s_1$ of the automorphism $\overline{s}$ is equal to the degree twisting by $\alpha$, that is, $\sigma_\alpha$, and
\item the rational number $\alpha$ is not equal to $0$ or $1$.
\end{itemize}
Our goal is to establish the existence of 
\begin{itemize}
\item an $A_\infty$ algebra $(V,d',\overline{\mu'})$, 
\item an automorphism $\overline{s'}$ of $(V,d',\overline{\mu'})$, and
\item an $A_\infty$ isomorphism $F$ between $(V,d,\overline{\mu})$ and $(V,d',\overline{\mu'})$
\end{itemize}
such that 
\begin{itemize}
\item the differential $d'$ and the operation $\mu'_3$ vanish,
\item the arity one projection $s'_1$ of the automorphism $\overline{s'}$ is equal to the degree twisting by $\alpha$ and the arity two projection $s_2$ vanishes, and
\item the isomorphism $F$ is of the form $(\id_V, f_2,0,0\ldots)$ and intertwines $\overline{s}$ and $\overline{s'}$.
\end{itemize}
Let us begin.

The condition that $\overline{s}$ is an automorphism of $(V,0,\overline{\mu})$ means, among other conditions, that the following equation is satisfied:
\begin{multline*}
\sigma_\alpha \circ \mu_3 \pm s_2\circ (\id\otimes \mu_2) \pm s_2\circ (\mu_2\otimes \id)
\\=
\mu_3\circ (\sigma_\alpha^{\otimes 3}) \pm \mu_2\circ (\sigma_\alpha \otimes s_2) \pm \mu_2\circ(s_2\otimes \sigma_\alpha).
\end{multline*}
This is an equation of maps from $V^{\otimes 3}$ to $V$. Applying it to the elementary tensor $v\otimes w\otimes x$ we get
\begin{multline*}
\alpha^{|v|+|w|+|x|+1}\mu_3(v,w,x) \pm s_2(v\otimes \mu_2(w,x)) \pm s_2(\mu_2(v,w),x)
\\=
\alpha^{|v|+|w|+|x|}\mu_3(v,w,x) \pm \alpha^{|v|}\mu_2(v, s_2(w,x)) \pm \alpha^{|x|}\mu_2(s_2(v,w),x).
\end{multline*}
Gathering terms and dividing through by $\alpha^{|v|+|w|+|x|+1}-\alpha^{|v|+|w|+|x|}$ (well-defined because $\alpha$ is neither $0$ nor $1$) we get the equation
\begin{multline*}
\mu_3(v,w,x) \pm \frac{s_2(v\otimes \mu_2(w,x)) \pm s_2(\mu_2(v,w),x)}{\alpha^{|v|+|w|+|x|+1}-\alpha^{|v|+|w|+|x|}} 
\\
\mp \frac{\mu_2(v, s_2(w,x))}{\alpha^{|w|+|x|+1}-\alpha^{|w|+|x|}}
\mp \frac{\mu_2(s_2(v,w),x)}{\alpha^{|v|+|w|+1}-\alpha^{|v|+|w|}}=0.
\end{multline*}
If we define a map $f_2:V^{\otimes 2}\to V$ as 
\[f_2(v,w) = \frac{s_2(v,w)}{\alpha^{|v|+|w|}(\alpha-1)},\]
then this equation can be further rewritten
\begin{multline}
\label{equation: A-infinity example}
\mu_3(v,w,x)
\pm f_2(v\otimes \mu_2(w,x) \pm f_2(\mu_2(v,w),x)
\\\mp \mu_2(v, f_2(w,x))
\mp \mu_2(s_2(v,w),x)=0.
\end{multline}
We set this equation aside, returning to it a bit later.

Now the sequence of maps $(\id_V,f_2,0,0,\ldots)$ can be extended to be an endomorphism $F$ of the tensor coalgebra $T^c(V[1])$ on a degree shifted copy of $V$.
In fact, $F$ is an \emph{automorphism} of this tensor coalgebra because its first component $\id_V$ is an automorphism of $V$.
We don't need a full description of the inverse $F^{-1}$ but it is easy to compute that in components, it begins $(\id_V, -f_2,\ldots)$ --- this is very much like inverting a formal power series with vanishing constant term.

Now we can define a new $A_\infty$ algebra structure $(V,d',\overline{\mu'})$ by conjugating by this automorphism of the tensor coalgebra. 
Using $d=0$ and the descriptions of $F$ and $F^{-1}$ we can calculate the following:
\begin{align*}
d'&=\id_V \circ d \circ \id_V = 0;\\
\mu'_2 &= \id_V \circ \mu_2 \circ (\id_V)^{\otimes 2} + \text{ terms involving }d
\\&= \mu_2;\\
\mu'_3&= \id_V \circ \mu_3 \circ (\id_V)^{\otimes 3} 
\\&\quad{}\pm  f_2 \circ (\id_V \otimes \mu_2) \pm f_2\circ (\mu_2\otimes \id_V)
\\&\quad{}\pm \mu_2 \circ (-f_2\otimes \id_V) \pm \mu_2\circ(\id_V\otimes (-f_2))
\\&\quad{}+\text{ terms involving }d.
\end{align*}
The right hand side here is zero by Equation~\eqref{equation: A-infinity example}, so $\mu'_3$ vanishes.

At this point we have established the existence of the desired $A_\infty$ algebra structure.
It is also true by construction that $F$ is an $A_\infty$ isomorphism between $(V,0,\mu)$ and $(V,0,\mu')$ of the desired form.

In order for an automorphism $\overline{s'}$ of $(V,0,\mu')$ to be intertwined with $\overline{s}$ by $F$, it is necessary for $\overline{s'}$ to be obtained from $\overline{s}$ via conjugation by $F$.

Again we calculate the first terms explicitly:
\begin{align*}
s'_1 &= \id_V\circ s_1\circ \id_V = s_1 = \sigma_\alpha;
\\
s'_2 &= 
\id_V\circ s_2 \circ (\id_V\otimes\id_V) 
+ f_2 \circ (s_1\otimes s_1)\circ (\id_V\otimes \id_V)
+ \id_V \circ s_1 \circ (-f_2)
\end{align*}
and applying $s'_2$ to the pair $(v,w)$ we get
\begin{align*}
s'_2(v,w) &= s_2(v,w) + \frac{\alpha^{|v|+|w|}}{\alpha^{|v|+|w|+1}-\alpha^{|v|+|w|}}s_2(v,w)
- \frac{\alpha^{|v|+|w|+1}}{\alpha^{|v|+|w|+1}-\alpha^{|v|+|w|}}s_2(v,w)
\\&=0.
\end{align*}
We conclude that $\overline{s'}$ has the desired form, concluding the example.

\end{example}

\subsection{Proof of the key lemma}
Now we will prove Lemma~\ref{lemma: key lemma} in generality. We proceed constructively in stages, assuming throughout that $n$ is at least $1$. 
For the remainder of the section, assume as given the data in the hypotheses of Lemma~\ref{lemma: key lemma}.
\begin{defi}
\label{defi: f}
We begin by defining a map $f_{n}:C_{n}(V)\to V$ as follows. 
On the homogeneous degree $n+N$ component of $C_{n}(V)$ (i.e., the component with homological degree $n$ in $C$ and degree $N$ in tensor powers of $V$) we act by $\frac{1}{\alpha^{N+n}-\alpha^{N}}$ times $s_{n}$ (the degree $n$ component of the automorphism $s$, which has the desired domain and codomain).  
The hypotheses on $\alpha$ ensure that the coefficient is well-defined.

Now define a linear map $f:C(V)\to V$ as
\[
f_i:C_i(V)\to V=\begin{cases}
\id_V & i=0\\
f_{n} & i=n\\
0 & \text{otherwise.}
\end{cases}
\]
\end{defi}
Then $f$ defines a map of coalgebras $F:C(V)\to C(V)$.

For the remainder of the section we will employ a simplified notation
\begin{notation}
Given maps $f$ and $g$ from $C(V)$ to $V$, we will write $f\compositionproduct g$ for the composition
\[
C(V)\xrightarrow{\Delta}C(C(V))\xrightarrow{C(g)}C(V)\xrightarrow{f}V
\]
and $f\compositionproduct_{(1)}g$ for the composition
\[
C(V)\xrightarrow{\Delta_{(1)}}
(C\compositionproduct_{(1)}C)(V)
\cong 
C(V;C(V))
\xrightarrow{C(V;g)}
C(V;V)\to 
C(V)\xrightarrow{f}V.
\]
where the unmarked arrow is induced by the fold map of $V$.
We call this \emph{schematic notation}.
\end{notation}
One reason for this notation is that $H=F\circ G$ is an equation of maps of coalgebras $C(V)\to C(V)$, if and only if $h=f\compositionproduct g$ is an equation of maps $C(V)\to V$.

\begin{lemma}
\label{lemma: f is invertible}
The map $F$ is invertible, and writing $f^{-1}$ for the projections of $F^{-1}$ to the cogenerators $V$ we have the following properties on its components:
\[
f^{-1}_{i}=\begin{cases} 
\id & i=0\\
0 & 1\le i <n\\
-f_{n}&i=n
\end{cases}
\]
(we make no claim for $i> n$).
\end{lemma}
\begin{proof}
We begin by establishing invertibility. 
Using schematic notation, $H=G\circ F$ is the identity map of $C(V)$ if and only if $h=g\compositionproduct f$ is the projection of $C(V)$ to $V$. 
The equation $h=g\compositionproduct f$ decomposes component by component as follows:
\begin{align*}
h_0 &= g_0\compositionproduct f_0\\
h_1 &= g_1\compositionproduct f_0 + \text{terms containing }g_i\text{ only having }i<1\\
\vdots&\qquad\vdots\\
h_j &= g_j\compositionproduct f_0 + \text{terms containing }g_i\text{ only having }i<j.
\end{align*}
we will solve this inductively with $h_0:C_0(V)\cong V\to V$ the identity of $V$ and $h_i=0$ for $i>0$.
Since $f_0=\id_V$ the only solution to the $h_0$ equation has $g_0=\id_V$ as well.
Then by induction, for the remaining equations to hold with $h_j=0$, $g_j$ must (and can) be chosen to be the negation of the remaining terms on the right hand side of the $j$th equation.
This establishes the existence of a one-sided inverse to $F$. 

We can perform a similar manipulation to establish the existence of an inverse on the other side, using the equations $H=F\circ G'$ and $h=f\compositionproduct g'$ and the componentwise decomposition
\begin{align*}
h_0 &= f_0\circ g'_0\\
h_1 &= f_0\circ g'_1 + \text{terms containing }g'_i\text{ only having }i<1\\
\vdots&\qquad\vdots\\
h_j &= f_0\circ g'_j + \text{terms containing }g'_i\text{ only having }i<j.
\end{align*}
In this case we use $\circ$ instead of $\compositionproduct$ because $f_0$ is concentrated in arity one.
Again, $g'_0$ must be $\id_V$ while $g'_j$ (for $j\ge 1$) can and must be taken as the negation of the remaining terms of the right hand side of the $h_j$ equation. This suffices to define $g'_j$ inductively which in turn establishes two-sided invertibility of $F$.

We have along the way also already established that $f^{-1}_0=\id_V$. 
To get the remaining formulas for $F^{-1}$ under projection and in components $1\le j<n$, we examine the terms of this second sequence of equations more closely.
Every term in the equation for $g'_j$ contains at least one $f_i$ with $1\le i<n$, and thus vanishes by assumption---thus $g'_j=(f^{-1})_j$ vanishes as well.

For $j=n$, precisely one term in the equation for $(f^{-1})_n=g'_n$ survives, namely $-f_n\compositionproduct g_0=-f_n$.
\end{proof}
Now we can define the structures $M'$ and $S'$ via conjugation.
\begin{defi}
\label{defi: construction of M' and S'}
Given the construction of Definition~\ref{defi: f} and Lemma~\ref{lemma: f is invertible}, define a map $M':C(V)\to C(V)$ of degree $-1$ and a map $S':C(V)\to C(V)$ of degree $0$ as follows.
\begin{align*}
M'&=FMF^{-1}&S'&=FSF^{-1}.
\end{align*}
\end{defi}
\begin{lemma}
\label{lemma: M' and S' properties}
The data of Definition~\ref{defi: construction of M' and S'} satisfies the following properties:
\begin{itemize}
\item $M'$ is a $P_\infty$-algebra structure on $V$, 
\item $S'$ is a $P_\infty$-automorphism of $(V,M')$, and
\item $F$ is a $P_\infty$-isomorphism between $(V,M)$ and $(V,M')$.
\end{itemize}
\end{lemma}
\begin{proof}
Since $M$ squares to zero, the map $M'$ is a degree $-1$ square zero coderivation of the cofree $C$-coalgebra, and thus a $P_\infty$-algebra structure on $V$. 
Likewise, from the properties of $S$, the map $S'$ is a differential graded $C$-coalgebra map with respect to $M'$ and thus a $P_\infty$ automorphism of the $P_\infty$-algebra $(V,M')$. 
By construction, the coalgebra map $F$ intertwines the coderivations $M$ and $M'$ and the automorphisms $S$ and $S'$.
\end{proof}
It remains to be seen that $m'$ and $s'$ have the promised description. 

\begin{lemma}
\label{lemma: M' has right shape}
The $P_\infty$-algebra structure $M'$ has component $m'_1$ equal to $m_1$. 
The component $m'_i$ vanishes on $C_i(V)$ for $i=0$ and for $i$ in the range $2\le i< n+2$.
\end{lemma}
\begin{proof}
First note that the composition $m': C(V)\to V$ has two contributions. 
There is a part from $m$ of the form
\begin{equation}
\label{equation: m', m part}
C(V)\to (C\compositionproduct_{(1)} C)\compositionproduct C(V)
\xrightarrow{f^{-1}} (C\compositionproduct_{(1)} C)(V)\xrightarrow{m}C(V)\xrightarrow{f}V
\end{equation}
where the first map is induced by the comonoid map and linearized comonoid map of $C$. 
The other part comes from the internal differential of $C$ and is of the form
\begin{equation}
\label{equation: m', d_C part}
C(V)\to C\compositionproduct C(V)\xrightarrow{f^{-1}}C(V)\xrightarrow{d_C}C(V)\xrightarrow{f}V.
\end{equation}
We have the following facts about these compositions:
\begin{itemize}
\item the decomposition map of $C$ respects degree, 
\item the $0$th components of $f$ and $f^{-1}$ are the identity,
\item the $i$th components of $f$ and $f^{-1}$ vanish for $1\le i< n$, 
\item the operation $m_0$ vanishes, and 
\item the differential $d_C$ reduces degree by $1$ and vanishes on $C_1$.
\end{itemize}
These facts immediately imply that the $i$th component of the composition~\eqref{equation: m', m part} coincides with $m_i$ for $0\le i<n+1$ and that the $i$th component of the composition~\eqref{equation: m', d_C part} vanishes for $i$ in the same range.
It remains to show that $m'_{n+1}=0$, for which we will need a more complicated argument.

Let us review the condition that $s$ is a an automorphism of $(V,m)$.
This is the condition that $MS=SM$, which can be written as the equality of the composition
\begin{align}
\label{equation:MS}
C(V)&\to C(C(V))\xrightarrow{s}C(V)\xrightarrow{m}V;\\
\intertext{and the sum of the compositions}
\label{equation: SM m part}
C(V)&\to C(V;C(V))\xrightarrow{m}C(V;V)\to C(V)\xrightarrow{s} V
\\
\intertext{and}
\label{equation: SM d_C part}
C(V)&\xrightarrow{d_C} C(V)\xrightarrow{s}V.
\end{align}
We focus first on the $n+1$st component of the composition~\eqref{equation:MS}. 
We know $m_i$ vanishes for $i=0$ and $2\le i < n+1$. 
Therefore the only terms that survive use $m_1$ or $m_{n+1}$.
By Lemma~\ref{lemma: technical compatibility between f and linearization when components vanish.}, the $m_1$ term is
\begin{multline*}
C(V)\xrightarrow{\sigma_{\alpha}} C(V)\xrightarrow{} C_1(V;C_n(V))\xrightarrow{s_n\sigma_\alpha^{-1}}C_1(V;V)\to C_1(V)\xrightarrow{m_1}V. 
\end{multline*}
The other term in~\eqref{equation:MS} uses first $s_0=\sigma_\alpha$ on each argument and then applies $m_{n+1}$.

There are only two non-vanishing terms in the compositon~\eqref{equation: SM m part}. 
One involves applying first $m_{n+1}$ and then $s_0$. 
Since $C_0$ is trivial, this term can be rewritten
\[
C_{n+1}(V)\xrightarrow{m_{n+1}}V\xrightarrow{s_0}V.
\]
Moreover, because $m_{n+1}$ is of homological degree $n$, this is $\alpha^n$ times
\[
C_{n+1}(V)\xrightarrow{s_0=\sigma_\alpha}C_{n+1}(V)\xrightarrow{m_{n+1}}V.
\]
The other term involves acting first by $m_1$ and then $s_n$.
We can write $s_n$ as the composition $(s_n\sigma_\alpha^{-1})\sigma_\alpha$ and then $\sigma_\alpha$ commutes with $m_1$ which is homological degree $0$.

Finally, the composition~\eqref{equation: SM d_C part} has only a single term, first acting by $d_C$ and then by $s_n$. Again $(s_n=s_n\sigma_\alpha^{-1})\sigma_\alpha$, and since $d_C$ acts only on the $C$ factor, this is the composition where first we apply $\sigma_\alpha$, then $d_C$, and finally $s_n\sigma_\alpha^{-1}$. 

So schematically we can write the equation as follows:
\begin{align*}
(m_1\compositionproduct_{(1)} s_n\sigma_\alpha^{-1})\sigma_\alpha + m_{n+1}\sigma_\alpha
&=
\alpha^n m_{n+1}\sigma_\alpha+ (s_n\sigma_\alpha^{-1}\compositionproduct_{(1)}m_1)\sigma_\alpha
\\
&+(s_n\sigma_\alpha^{-1})d_C\sigma_\alpha.
\end{align*}
By construction the terms we are calling $s_n\sigma_\alpha^{-1}$ are equal to $(\alpha^{n}-1)f_n$.
That is, on homogeneous components of $C_n(V)$ of total homological degree $n+N$ (thus degree $N$ in tensor powers of $V$) we have
\[
s_n\sigma_\alpha^{-1}=\frac{1}{\alpha^N}s_n=(\alpha^{n}-1)\frac{1}{\alpha^{N+n}-\alpha^{N}}s_n=(\alpha^{n}-1)f_n.
\]
Then our schematic equation becomes
\[
(1-\alpha^n)(-m_1\compositionproduct_{(1)} f_n + m_{n+1} +f_n\compositionproduct_{(1)}m_1+f_nd_C)=0.
\]
As long as $\alpha^n\ne 1$, we can divide by $1-\alpha^n$ and use Lemma~\ref{lemma: f is invertible} to get the schematic equation
\[
m_1\compositionproduct_{(1)} f^{-1}_n + m_{n+1} +f_n\compositionproduct_{(1)}m_1 + f_n d_C=0.
\]
We claim that the left hand side of this equation is $m'_{n+1}$, which will complete the proof.

We return to the expressions~\eqref{equation: m', m part} and~\eqref{equation: m', d_C part} defining $m'_{n+1}$. 
Some of the terms in~\eqref{equation: m', m part} do not vanish from the conditions on $m$ and $f$. 
We classify these into three kinds.
The first kind is made up of compositions of $m_1$ with $f^{-1}_{n}$ and $f_0$, the second kind has only one member, the term $m_{n+1}$, and the third kind is made up of compositions of $f_{n}$ with $m_{1}$. 
For~\eqref{equation: m', d_C part}, the vanishing conditions on $f$ and $d_C$ imply that the only surviving term must be the composition of $f_n$ with $d_C$.
Essentially by Lemma~\ref{lemma: technical compatibility between f and linearization when components vanish.}, we can write the overall calculation in our schematic pidgin as
\[
m'_{n+1} = m_1\compositionproduct_{(1)}f^{-1}_n+m_{n+1} + f_n\compositionproduct_{(1)}m_1 + f_nd_C
\]
which is what we got from compatibility of $m$ and $s$.
\end{proof}
\begin{lemma}
The $P_\infty$-automorphism $S'$ has component $s'_0:C_0(V)\to V$ equal to $\sigma$ and $s'_i$ vanishes on $C_i(V)$ in the range $1\le i\le n+1$.
\end{lemma}
\begin{proof}
The composition $s':C(V)\to V$ takes the form
\[
C(V)\to C(C(C(V)))\xrightarrow{f^{-1}}C(C(V))\xrightarrow{s}C(V)\xrightarrow{f}V,
\]
where the first map is induced by the decomposition of $C$.
As in Lemma~\ref{lemma: M' has right shape}, the vanishing conditions for the degrees of $C$ in which the maps $f$ and $f^{-1}$ are supported ($0$ or at least $n$) and the fact that $f_0=f^{-1}_0=\id_V$ imply that $s'_i=s_i$ for $i<n$.
For $s'_n$, we have the schematic equation
\[
s'_n = f_n\compositionproduct s_0 + s_n + s_0\compositionproduct f^{-1}_n
\]
and so acting on the homogeneous component of $C_n(V)$ of total homological degree $n+N$, we have the equality
\[
s'_n=
\frac{\alpha^{N}}{\alpha^{N+n}-\alpha^{N}}s_n + s_n - \frac{\alpha^{N+n}}{\alpha^{N+n}-\alpha^{N}}s_n=0,
\]
as desired.
\end{proof}
This concludes the proof of Lemma~\ref{lemma: key lemma}.
\subsection{Using the key lemma}
As the output of the lemma yields the input data with an increased index $n$, we can recursively arrive at the following.
\begin{equation}
\label{eq:transfinite sequence}
(V,m^{[1]},s^{[1]})\xrightarrow{f^{[1,2]}}(V,m^{[2]},s^{[2]})\xrightarrow{f^{[2,3]}} \cdots\xrightarrow{f^{[n-1,n]}} (V,m^{[n]},s^{[n]})\xrightarrow{f^{[n,n+1]}}\cdots
\end{equation}
where 
\begin{enumerate}
\item the data $m^{[j]}$ is a $P_\infty$ structure on $V$ which vanishes on $C_i(V)$ for $i=0$ and $i$ in the range $2\le i< j+1$,
\item the data $s^{[j]}$ is a $P_\infty$-automorphism of $(V,m^{[j]})$ which is equal to $\sigma_{\alpha}$ on $V$ and vanishes on $C_i(V)$ in the range $1\le i<j$, and
\item the data $f^{[j,j+1]}$ is a $P_\infty$-isomorphism from $(V,m^{[j]})$ to $(V,m^{[j+1]})$ intertwining $s^{[j]}$ and $s^{[j+1]}$ which is equal to the identity on $V$ and vanishes on $C_i(V)$ for $i\notin \{0,j\}$.
\end{enumerate}
We can continue this procedure up to the smallest $n$ such that $\alpha^n-1$ is not a unit. If $\alpha^n-1$ is always a unit, then~\eqref{eq:transfinite sequence} is an infinite sequence.
\begin{lemma}
\label{lemma: transfinite well defined}
If $\alpha^n-1$ is a unit for all $n$, the transfinite composition 
\[f^{[1,\omega]}\coloneqq\cdots \circ f^{[n,n+1]}\circ f^{[n-1,n]}\circ\cdots \circ f^{[1,2]}\] is well-defined.
In particular, the component $f^{[1,\omega]}_n:C_n(V)\to V$ is equal to the $C_n(V)\to V$ component of the finite composite
\[
f^{[n,n+1]}\circ f^{[n-1,n]}\circ\cdots \circ f^{[1,2]}.
\]
\end{lemma}

\begin{proof}
Because the $0$th component of $f^{[m,m+1]}$ is the identity and all other components with index less than $m$ are $0$, the composition of $F^{[m,m+1]}$ with an arbitrary $P_\infty$ map $G:W\to V$ has components $C_n(W)\to V$ equal to those of $G$ for all indices $n$ less than $m$.
\end{proof}
\begin{remark}
We were not able to find a comprehensible closed form expression for the coefficients involved in the transfinite composition, which seem to involve a summation over some special classes of decorated trees.
\end{remark}

Now we are ready to prove the main theorem.

\begin{proof}[Proof of Main Theorem]
By the homological perturbation lemma (Theorem~\ref{theorem:HPL}) there is a $P_\infty$-structure $m^{[1]}$ on $H(A,d)$, along with $P_\infty$ quasi-inverses 
\[
\iota:(H(A,d),m^{[1]})\rightleftarrows (A,d,m):\pi.
\]
Since these are quasi-inverses the composition $\pi\iota$ is homotopic to the identity of $H(A,d)$. But since $H(A,d)$ has no differential, this means $\pi\iota$ is equal to the identity. 
Define $s^{[1]}$ as the composition of the $P_\infty$ morphisms $\pi$, $\hat{\sigma}$, and $\iota$; then $s^{[1]}$ is automatically a $P_\infty$ automorphism of $(H(A,d), m^{[1]})$.
Its zeroth component is given by $\pi_0\hat{\sigma}\iota_0$. 
Since $\iota$ is a $P_\infty$-morphism, in particular $\iota_0$ lands in the cycles of $A$ with respect to $d$, and then up to boundary terms, $\hat{\sigma}\iota_0=\iota_0\sigma$. But then $\pi_0$ kills boundary terms so that 
\[
\pi_0\hat{\sigma}\iota_0=\pi_0\iota_0\sigma=\sigma.
\]
Then $m^{[1]}$ and $s^{[1]}$ are precisely the input data necessary for Lemma~\ref{lemma: key lemma} and Lemma~\ref{lemma: transfinite well defined}. If $\alpha^k-1$ is a unit for all $k\leq n$, then the finite composition
\[
f^{[1,n]}:= f^{[n-1,n]}\circ\cdots \circ f^{[1,2]}
\]
is an isomorphism between $m^{[1]}$ and a $P_\infty$-structure whose first $n$ components vanish. If $\alpha^k-1$ is always a unit, the transfinite composition $f^{[1,\omega]}$ constructed in Lemma~\ref{lemma: transfinite well defined} is an isomorphism between $m^{[1]}$ and the $P_\infty$-structure $m_*$.
\end{proof}

\begin{remark}
\label{remark: generalization of homotopy retract}
This proof relies on the following consequences of the homological perturbation lemma for the base case:
\begin{enumerate}
	\item the chain complex $(H(A,d),0)$ supports a $P_\infty$-algebra structure $m^{[1]}$ with first component which is induced by $m$ along with a quasi-isomorphism $(H(A,d),0,m^{[1]})\to (A,d,m)$, and
	\item the $P_\infty$-algebra $(H(A,d),0,m^{[1]})$ supports a $P_\infty$-automorphism $s^{[1]}$ with first component $\sigma$ (induced by $\hat{\sigma}$).
\end{enumerate}
Any hypotheses guaranteeing these two conditions is sufficient for the argument. 
\end{remark}

\begin{remark}
Bruno Vallette has pointed out to us that the core of the argument presented in this section is more or less a computation in the convolution preLie algebra $\operatorname{Hom}_{\mathbb{S}}(C,\operatorname{End}_{H(A,d)})$ (see, e.g.,~\cite[\S~6.4.2]{LodayVallette:AO} and~\cite[\S~5]{DotsenkoShadrinVallette:PLDT}). 
It looks like an interesting and possibly illuminating exercise to rewrite the proof using this observation. 
\end{remark}

\section{A variant of the main theorem}
\label{section: variant theorem}

We are interested in examples where the homology of the $P$-algebra $A$ is concentrated in degree divisible by $c$ for some integer $c$ (an example is given by the cohomology of $\mathbb{CP}^n$ which is concentrated in even degrees). In such a situation we can ``ignore'' the zero cohomology groups and our main theorem becomes the following.

\begin{theorem}\label{theo: variant}
Let $(A,d,m)$ be a differential graded $P$-algebra in $\groundring$-modules such that the chain complex $(H(A,d),0)$ can be written as a homotopy retract of $(A,d)$. 
Assume further that the homology $H(A,d)$ is concentrated in degrees divisible by $c$.
Let $\alpha$ be a unit in $\groundring$ and let $\hat{\sigma}$ be an endomorphism of $(A,d,m)$ such that the induced map on $H_{cn}(A,d)$ is multiplication by $\alpha^n$. 
\begin{itemize}
\item If $\alpha^k-1$ is a unit of $\groundring$ for $k\leq n$, then $(A,d,m)$ is $cn$-formal as a $P$-algebra.
\item If $\alpha^k-1$ is a unit of $\groundring$ for all $k$, then $(A,d,m)$ is formal as a $P$-algebra.
\end{itemize}
\end{theorem}

Observe that if $\alpha$ has a $c$-th root in our ring $\groundring$, then this theorem is our main theorem with $\alpha$ replaced by $(\alpha)^{1/c}$. In general this variant is proved by adapting the proof of Lemma \ref{lemma: key lemma} in an obvious manner.

\section{Applications}
\label{sec: applications}

\subsection{Hodge theory}\label{subsection:Hodge}

Let us reinterpret the classical Deligne--Griffiths--Morgan--Sullivan~\cite{DeligneGriffithsMorganSullivan:RHTKM} result in the light of the present paper.  
This has already been done indirectly by Sullivan~\cite[\S12]{Sullivan:ICT}. 
There he argues that complex formality of K\"ahler manifolds (proven by other means in~\cite{DeligneGriffithsMorganSullivan:RHTKM}) implies that degree twisting automorphisms lift to the chain level. 
Then rational degree twisting automorphisms lift to the chain level, and therefore the rational homotopy type of a K\"ahler manifold is formal.

We outline a more direct proof along the same lines, sketching a way to obtain real-valued chain level lifts of the degree twisting automorphism. Arguably this is not the most natural way to prove such results so we do not provide full details. We denote by $\mathrm{MHS}$ the abelian category of rational mixed Hodge structures. An object of $\mathrm{MHS}$ is a triple $(H,W,F)$ where $H$ is a $\mathbb{Q}$-vector space, $W$ is an increasing filtration on $H$ (the weight filtration) and $F$ is a decreasing filtration on $H\otimes_{\mathbb{Q}}\mathbb{C}$ (the Hodge filtration). We make the following conjecture.

\begin{conj}\label{conjecture : Hodge model}
Let $X$ be a complex algebraic variety. There exists a commutative algebra $A^*(X)$ in the category of cochain complexes in $\mathrm{MHS}$ which represents the rational homotopy type of $X$ and such that the induced weight filtration on $H^*(A(X))\cong H^*(X,\mathbb{Q})$ is the weight filtration of the mixed Hodge structure constructed by Deligne in \cite{Deligne:THII, Deligne:THIII}.
\end{conj}

Let us make precise what we mean by ``represents the rational homotopy type''. There is a forgetful symmetric monoidal functor from the category $\mathrm{MHS}$ to the category of rational vector spaces. We can apply this functor to the model $A^*(X)$ and we obtain a commutative differential graded algebra over $\mathbb{Q}$ and we require that the resulting object is quasi-isomorphic to Sullivan's commutative differential graded algebra of piecewise polynomial differential forms.

Let us give some ideas on how one should be able to prove such a conjecture. In \cite{CiriciHorel:MHSFSMF}, Cirici and the second author explained how one can view a model for the singular cochains of $X$ as a commutative algebra in the $\infty$-category of chain complexes of real mixed Hodge structures. Then we believe that a rigidification result similar to the one proved in Hinich \cite{Hinich:RAM} should be true in that context and we can actually represent this commutative algebra in the $\infty$-category by a strict commutative algebra.

We assume that this conjecture is correct until the end of the subsection. The abelian category of real mixed Hodge structure is a Tannakian category. As a fiber functor we can take the functor
\[(H,F,W)\mapsto \oplus_n\mathrm{gr}_n^W(H)\]
that sends a mixed Hodge structure to the associated graded of the weight filtration. This functor is isomorphic to the forgetful functor $(H,F,W)\mapsto H$. An explicit isomorphism was given by Deligne over the real numbers in \cite[1.2.11]{Deligne:THII}. It was observed in \cite[Lemma 4.4.]{CiriciHorel:MHSFSMF} that, for abstract reasons, there must exist such an isomorphism over $\mathbb{Q}$ as well, although it is not explicit. By definition, this fiber functor factors through the Tannakian category of graded vector spaces. Therefore, by Tannaka duality, if we denote by $\mathrm{G}_{\mathrm{MHS}}$ the Tannakian Galois group of mixed Hodge structure, there exists a maps of affine group schemes over $\mathbb{Q}$
\[i:\mathbb{G}_m\to \mathrm{G}_{\mathrm{MHS}}\]
We refer the reader to~\cite{Goncharov:HHSHC} for more details about this. According to our conjecture, there is an action of $\mathrm{G}_{\mathrm{MHS}}(\mathbb{Q})$ on our model $A^*(X)$ of the rational homotopy type of $X$. We can restrict this action along $i$ and we get an action of $\mathbb{Q}^{\times}=\mathbb{G}_m(\mathbb{Q})$ on $A^*(X)$. We have the following fact about this action.

\begin{prop}
Let $X$ be an algebraic variety over $\mathbb{C}$. If $H^k(X,\mathbb{R})$ is a cohomology group of $X$ whose mixed Hodge structure is pure of weight $n$, then the action of $x\in\mathbb{Q}^{\times}$ on $H^k(X)$ is given by multiplication by $x^n$.
\end{prop}

\begin{proof}
This can be found in \cite[Paragraph 2.1.5.1, p.25]{Deligne:THII}.
\end{proof}

Hence, applying our main result to the commutative algebra over the rational numbers $A^*(X)$, we can prove the following theorem.

\begin{theorem}
Let $X$ be a smooth projective complex variety, or more generally a variety satisfying the property that $H^k(X,\mathbb{R})$ is a pure Hodge structure of weight $k$ for all $k$. Then there exists a model $A^*(X)$ for the rational homotopy type of $X$ that is formal.
\end{theorem}

It should be noted that the approach outlined in this section is dependent on Conjecture \ref{conjecture : Hodge model}. However the result is true regardless of this conjecture as was proved, in various degrees of generality, in \cite{DeligneGriffithsMorganSullivan:RHTKM,dupont:PFAC,CiriciHorel:MHSFSMF}. It should also be noted that if instead the variety $X$ has the property that $H^k(X,\mathbb{R})$ is a pure Hodge structure of weight $\alpha k$ for all $k$, where $\alpha$ is a fixed non-zero rational number, then, using Theorem \ref{theo: variant}, we also get a formality result that recovers the one of \cite{CiriciHorel:MHSFSMF}. 

\subsection{Formality of the little disks operad}

Following Petersen, we can prove formality of the little disks operad with rational coefficients. Let us denote by $\mathcal{D}_2$ the topological little disks operad. Petersen observes that for any $\alpha$ in $\mathbb{Q}^{\times}$, there exists an automorphism of $C_*(\mathcal{D}_2,\mathbb{Q})$ that induces the grading automorphism $\sigma_{\alpha}$ on the homology (see the proof of the main Theorem of \cite{Petersen:GTFLD}). We can give an alternative proof of the Proposition in \cite{Petersen:GTFLD} using our Main Theorem. In that case the ring of coefficient is $\mathbb{Q}$ and the operad we consider is the colored operad that controls the structure of a single colored operad. Note that this proof is closer in spirit to the intuition developed in the last section of \cite{Petersen:GTFLD}.

\subsection{Complement of subspace arrangements}

In this subsection, we denote by $K$ a finite extension of $\QQ_p$. The residue field of the ring of integers of $K$ is isomorphic to $\mathbb{F}_q$ for $q$ some power of $p$. We denote by $\ell$ a prime number different from $p$ and we denote by $h$ the order of $q$ in $\FF_{\ell}^{\times}$. 

For us, a complement of a hyperplane arrangement over a field $L$ is the complement of a finite collection of affine hyperplanes in $\mathbb{A}^n_L$ viewed as a scheme over $L$. We say that a complement of hyperplane arrangements $X$ over the complex numbers is defined over $K$ if there exists an embedding $\iota:K\to\mathbb{C}$ and a complement of a hyperplane arrangement over $K$ denoted $\mathcal{X}$ such that $\mathcal{X}\times_K\mathbb{C}$ is isomorphic to $X$.

\begin{theorem}\label{theo : complement of hyperplane arrangements}
Let $X$ be a complement of a hyperplane arrangement over the complex numbers. Assume that $X$ is defined over $K$. Then the dg-algebra $C^*(X_{an},\ZZ_\ell)$ is $(h-1)$-formal
\end{theorem}

\begin{proof}
Indeed in that case, by standard comparison results in \'etale cohomology, we have a quasi-isomorphism of dg-algebras
\[C^*(X_{an},\ZZ_\ell)\simeq C^*_{\textit{\'et}}(\mathcal{X}_{\overline{K}},\ZZ_\ell)\]
where $\mathcal{X}$ is the complement of a hyperplane arrangement defined over $K$ that exists by assumption. Let $\sigma$ be a Frobenius lift, i.e. an element of $\mathrm{Gal}(\overline{K}/K)$ that maps to a generator of $\mathrm{Gal}(\overline{\FF}_q/\FF_q)$. Then the action of $\sigma$ on $H^n_{\textit{\'et}}(\mathcal{X}_{\overline{K}},\ZZ_\ell)$ is given by multiplication by $q^n$ (see \cite[Theorem 1']{Kim:WCGAHA}). We are thus precisely in the situation of our main theorem. The ring of coefficient is $\mathbb{Z}_\ell$ and the operad $P$ is the associative operad.
\end{proof}

\begin{remark}
This theorem is the codimension $1$ version of \cite[Theorem 8.11]{CiriciHorel:ECPFTC}. Note that, in contrast to \cite{CiriciHorel:ECPFTC}, we do not make the assumption that $X_{an}$ is simply connected. The higher codimension version of \cite[Theorem 8.11]{CiriciHorel:ECPFTC} can also be proved using the methods of the current paper. It should also be noted that we obtain a formality result over $\ZZ_\ell$ and not just over $\FF_\ell$ as in \cite{CiriciHorel:ECPFTC}.
\end{remark}

\begin{remark}
Note that the condition of being defined over $K$ cannot be dropped. Indeed for each prime $\ell$, Matei in \cite{Matei:MPCHA}, gives an example of a hyperplane arrangements in $\mathbb{C}^3$ whose complement has a non-trivial Massey products in $H^2(-,\FF_\ell)$. However, the equations of his hyperplanes involve $\ell$-th roots of unity. If a $p$-adic field $K$ has $\ell$-th roots of unity then the residue field $\FF_q$ must have $\ell$-th roots of unity as well and this implies that $\ell$ divides $q-1$. But in that case $q$ is congruent to $1$ modulo $\ell$ and therefore the previous theorem is an empty statement. This is in sharp contrast with the case of rational coefficients where all complements of hyperplane arrangements are formal. 
\end{remark}

\begin{remark}
On the other hand, if the hyperplane arrangement is defined over $\mathbb{Q}$, then it is defined over $\QQ_p$ for every $p$. We are then free to pick $p$ such that $p$ is of order $(\ell-1)$ in $\FF_\ell^\times$ and we obtain $(\ell-2)$-formality for such arrangements. Example of such arrangements are the $A_n$, $B_n$, $C_n$ and $D_n$ arrangements.
\end{remark}

We also have a similar result for complements of toric arrangements. We first recall the relevant definition. A \emph{character} of $(\mathbb{C}^*)^d$ is an algebraic group homomorphism $(\mathbb{C}^*)^d\to\mathbb{C}^*$. It is straightforward to check that any character is of the form
\[(z_1,\ldots,z_d)\mapsto z_1^{n_1}\ldots z_d^{n_d}\]
with $n_1,\ldots,n_d$ a sequence of integers. Given a character of $(\mathbb{C}^*)^d$ and a non-zero complex number $a$, we denote by $H_{\chi,a}$ the subvariety of $(\mathbb{C}^*)^d$ defined by the equation
\[\chi(z_1,\ldots,z_d)=a\]

\begin{defi}
A complement of a toric arrangement is an open subspace of $(\mathbb{C}^*)^d$ of the form
\[(\mathbb{C}^*)^d-\bigcup_{i=1}^nH_{\chi_i,a_i}\]
where each $\chi_i$ is a character and each $a_i$ is a non-zero complex number.
\end{defi} 

We say that  a complement of a toric arrangement $X$ is defined over $K$ if there exists an embedding $\iota:K\to\mathbb{C}$ such that the coefficient $a$ in the equation of each $H_{\chi,a}$ is in the image of $\iota$. Given such a choice of $\iota$ we can construct a variety $\mathcal{X}$ over $K$ given as the open complement in $\mathbb{G}_m^d$ of the closed subsets $H_{\chi,a}$ (or more precisely $H_{\chi,\tilde{a}}$ where $\tilde{a}$ is the preimage of $a$) and we have
\[X=\mathcal{X}\times_{\mathrm{Spec}(K)}\mathrm{Spec}(\mathbb{C})\]

\begin{prop}
Let $X$ be a complement of a toric arrangement that is defined over $K$. Then the dg-algebra $C^*(X,\ZZ_\ell)$ is $(h-1)$-formal.
\end{prop}

\begin{proof}
Again, by comparison between \'etale and singular cohomology, it suffices to prove that $C^*_{\textit{\'et}}(\mathcal{X}_{\overline{K}},\ZZ_\ell)$ is formal. We claim that for any choice of Frobenius lift $\sigma$ in $\mathrm{Gal}(\overline{K}/K)$ the action of $\sigma$ on $H^n_{\textit{\'et}}(\mathcal{X}_{\overline{K}},\ZZ_\ell)$ is given by multiplication by $q^n$. Indeed, by \cite{dAntonioDelucchi:MTA} the cohomology of the complement of a toric arrangement is torsion free, it follows that it is enough to prove that $\sigma$  acts on $H^n_{\textit{\'et}}(\mathcal{X}_{\overline{K}},\QQ_\ell)$ by multiplication by $q^n$. The proof of the analogous statement in the Hodge case (instead of the \'etale case) is done in \cite[2.2]{Looijenga:CMM}. We can therefore conclude as in the proof of Theorem \ref{theo : complement of hyperplane arrangements}.
\end{proof}

\begin{remark}
This result about formality of complements of hyperplane arrangements and toric arrangements with coefficients in $\mathbb{Z}_{\ell}$ is new as far as the authors know. Let us mention however, that the paper \cite{CallegaroDAdderrioDelucchiMigliorniPagaria:OSTPCATA} contains a result called integral formality for toric arrangements (see \cite[Theorem 7.4]{CallegaroDAdderrioDelucchiMigliorniPagaria:OSTPCATA}).  We believe that our result and the result of \cite{CallegaroDAdderrioDelucchiMigliorniPagaria:OSTPCATA} are independent. If we denote by $X$ the toric arrangement, our result is about formality of the singular cohomology algebra $C^*(X)$ whereas the result of \cite{CallegaroDAdderrioDelucchiMigliorniPagaria:OSTPCATA} identifies a certain $\mathbb{Z}$-subalgebra of the de Rham algebra of $X$ that is isomorphic to the integral cohomology of $X$. It is classical that the real singular cochains algebra is quasi-isomorphic to the de Rham algebra but such a result cannot be expected to hold integrally because the de Rham algebra is commutative whereas it is known that $C^*(X)$ is merely $E_\infty$ and cannot be strictified to a commutative differential graded algebra (the Steenrod operations are an obstruction to the existence of such a strictification). From the homotopical point of view the $E_\infty$-dg-algebra $C^*(X)$ is the most interesting object as it permits to reconstruct the homotopy type of $X$ by a famous theorem of Mandell (see \cite{Mandell:CHT}). Note however that our formality result is only about the associative differential graded algebra $C^*(X)$ and does not say anything about the $E_\infty$-structure.
\end{remark}

\begin{remark}
Using \'etale cohomology with coefficients in $\QQ_\ell$ instead of $\ZZ_{\ell}$, we can also prove formality (without bound) of cohomology of complements of hyperplane and toric arrangements with $\QQ_\ell$-coefficients. This gives an alternative proof of the results of Brieskorn and Dupont (proved respectively in \cite{Brieskorn:GT} and \cite{dupont:PFAC}). Let us mention however that the \'etale cohomology method only yields formality if the arrangement is defined over a $p$-adic field. The results we get are thus less general than those of Brieskorn and Dupont. If instead of \'etale cohomology we use Hodge theory as in subsection \ref{subsection:Hodge} we can completely recover the results of \cite{Brieskorn:GT} and \cite{dupont:PFAC}.
\end{remark}

\subsection{Coformality of configuration spaces}

Recall that, by the work of Quillen, the homotopy type of simply connected rational spaces is captured by a differential graded Lie algebra. One says that a space is coformal if this Lie algebra is formal. By \cite[Corollary 1.2]{Saleh:NFICLF}, asking for a space $X$ to be coformal is equivalent to asking for the dg-algebra $C_*(\Omega X,\QQ)$ to be formal (by $\Omega X$ we mean a stricly associative model for the loop space of $X$). The advantage of this second definition is that it can be generalized to coefficient rings that are not $\QQ$-algebras.

We are interested in the coformality of the configuration space of $n$ distinct ordered points in $\mathbb{R}^d$ denoted $\mathrm{Conf}_n(\mathbb{R}^d)$. These spaces are known to be rationally coformal. We have the following theorem about coformality with $\ZZ_p$-coefficients.

\begin{theorem}
Let $d\geq 3$. Let $X=\mathrm{Conf}_n(\mathbb{R}^d)$. Let $p$ be a prime number. The dg-algebra $C_*(\Omega X,\ZZ_p)$ is $(p-2)(d-2)$-formal.
\end{theorem}

\begin{proof}
We can first replace $X$ by its $p$-completion which we will do implicitly from now on. In \cite{BoavidaHorel:FLDOPC}, an action of $\mathrm{GT}_p$, the $p$-complete Grothendieck--Teichm\"uller group is constructed on $X$. As a consequence of this action, it is shown in \cite[Proposition 8.2]{BoavidaHorel:FLDOPC} that for any unit $\alpha$ in $\ZZ_p$, there exists an automorphism $\alpha^{\sharp}$ of $X$ that acts by multiplication by $\alpha$ in homological degree $(d-1)$. Since this is the bottom non-vanishing homology group of $X$, using the Hurewicz isomorphism twice, we can identify this group with $H_{d-2}(\Omega X,\ZZ_p)$. By \cite[Theorem 2.3]{CohenGitler:LSCS}, we know that the homology of $\Omega X$ is torsion-free, concentrated in degree divisible by $(d-2)$ and is generated as an algebra by classes of degree $(d-2)$. Since the action of $\alpha^\sharp$ exists at the space level, it is compatible with the dg-algebra structure on $C_*(\Omega X,\ZZ_p)$ and we deduce that $\alpha^\sharp$ acts as multiplication by $\alpha^k$ in homological degree $k(d-2)$. We can pick an $\alpha$ whose residue modulo $p$ is a generator of the group of units of $\FF_p$. For such an $\alpha$, the number $\alpha^k-1$ is a unit in $\ZZ_p$ for $k\leq p-2$. We are thus precisely in the situation of Theorem \ref{theo: variant} with $P$ the associative operad and $R=\ZZ_p$.
\end{proof}
\section*{Acknowledgments}
We would like to thank Bruno Vallette and the anonymous referee for helpful comments.

\bibliography{references-2019}
\bibliographystyle{amsalpha}

\end{document}